\documentclass[11pt,a4paper,fleqn]{scrartcl}
\pdfoutput=1
\RequirePackage{amsthm,amsmath,natbib}
\RequirePackage{amssymb,amstext}
\RequirePackage{hypernat}
\usepackage[pdfpagemode=UseOutlines ,plainpages=false
,hypertexnames=false ,pdfpagelabels ,hyperindex=true,colorlinks=true]{hyperref}
	\makeatletter
	\Hy@breaklinkstrue
	\makeatother
\usepackage{color}
\definecolor{darkred}{rgb}{0.6,0.0,0.1}
\definecolor{darkgreen}{rgb}{0,0.5,0}
\definecolor{darkblue}{rgb}{0,0,0.5}
\hypersetup{colorlinks ,linkcolor=darkblue ,filecolor=darkgreen
,urlcolor=darkblue ,citecolor=black}

\usepackage{url}
\usepackage{verbatim}
\usepackage{ bm}

\usepackage[english]{babel}

\renewcommand{\cite}{\citet}
\usepackage{abrev-package}

\definecolor{dgreen}{rgb}{0,0.5,0}
\definecolor{dblue}{rgb}{0,0,0.9}
\definecolor{dred}{rgb}{0.6,0.0,0.1}
\definecolor{dgold}{rgb}{0.5,0.3,0.0}
\definecolor{dvio}{rgb}{0.6,0.3,0.5}
\definecolor{gray}{rgb}{0.5,0.5,0.5}

\newtheoremstyle{mysc}
  {3pt}
  {3pt}
  {\it}
  {}
  {\color{darkred}\sc}
  {.}
  {.5em}
  {}

\newtheoremstyle{myex}
  {10pt}
  {10pt}
  {\rm}
  {}
  {\color{darkred}\sc}
  {.}
  {.5em}
  {}

\theoremstyle{mysc}\newtheorem{prop}{Proposition}[section]
\theoremstyle{mysc}
\theoremstyle{mysc}\newtheorem{coro}[prop]{Corollary}
\theoremstyle{mysc}\newtheorem{theo}[prop]{Theorem}
\theoremstyle{mysc}\newtheorem{defin}{Definition}[section]
\theoremstyle{mysc}\newtheorem{lem}[prop]{Lemma}
\theoremstyle{myex}\newtheorem{rem}{Remark}[section]
\theoremstyle{myex}\newtheorem{example}{Example}[section]
\theoremstyle{myex}

\theoremstyle{mysc}\newtheorem{assA}{Assumption}

\theoremstyle{mysc}
\theoremstyle{mysc}
\theoremstyle{mysc}

\numberwithin{equation}{section}

\makeatletter%
\def\@fnsymbol#1{\ensuremath{\ifcase#1\or * \or \star \or 1 \or 2\or 3\or  , \or
g\or h\or i\else\@ctrerr\fi}}%
\makeatother%

\author{{\sc Christoph Breunig}\thanks{
Lehrstuhl f\"ur Statistik, Abteilung Volkswirtschaftslehre, L7, 3-5, 68131 Mannheim, Germany, e-mail: \url{	cbreunig@staff.mail.uni-mannheim.de}} \\
{\small \it Universit\"at Mannheim} \and {\sc Jan Johannes}\thanks{Institut de statistique, biostatistique et sciences actuarielles (ISBA), Voie du Roman Pays 20, B-1348 Louvain-la-Neuve, Belgique, e-mail: \url{jan.johannes@uclouvain.be}}\\
{\small\it Universit\'e catholique de Louvain}}

\title{{\bf Adaptive  estimation of functionals in nonparametric instrumental regression.}}

\usepackage{srcltx}

\begin{document}
\date{\today}
\maketitle

\begin{abstract} We consider the problem of estimating the value $\ell(\sol)$ of a linear functional, where  the structural function $\sol$ models a nonparametric relationship in presence of instrumental variables.
 We propose a plug-in estimator which is based on a dimension reduction technique and additional thresholding. It is shown that this estimator is
consistent  and can attain the minimax optimal rate of convergence under additional regularity conditions. This, however, requires an optimal choice of  the dimension parameter $m$ depending on certain characteristics of  the structural function $\sol$ and the joint distribution of the regressor and the instrument, which are unknown in practice. We propose a fully data driven choice of $m$ which 
combines model selection and Lepski's method. We show that the  adaptive estimator attains the optimal rate of convergence up to a logarithmic factor. 
The theory in this paper is illustrated by considering  classical smoothness assumptions and we discuss examples such as pointwise estimation or estimation of
  averages of the structural function $\sol$.  
\end{abstract}

\begin{tabbing}
\> Optimal rates of convergence, Sobolev space,\\
\> finitely and infinitely smoothing operator,\\
\> Adaptive estimation, Model selection, Lepski's method.\\[.2ex]
\end{tabbing}
This work was supported by the DFG-SNF research group FOR916 and by the IAP research network no. P6/03 of the Belgian Government (Belgian Science Policy).
 \section{Introduction}
We consider estimation of the value of a linear functional of the structural function $\sol$ in a nonparametric instrumental regression model.
The structural function  characterizes the dependency of a response $Y$ on the variation  of an  explanatory random variable $Z$   by
  \begin{subequations}
	\begin{equation}
	\label{model:NP}
	Y = \varphi(Z) + U \quad\mbox{ with }\quad\Ex[U|Z]\ne 0
	\end{equation}
for some  error term $U$. In other words, the structural function equals not the conditional mean function of $Y$ given $Z$. In this model, however, a sample from $(Y,Z,W)$ is available,
where $W$ is a random variable, an
instrument,  such that 
	\begin{equation}
	\label{model:NP2}
	\Ex [U | W] = 0.
	\end{equation}
      \end{subequations}
Given some a-priori knowledge on the unknown structural function $\sol$, captured by a function class $\cF$, its  estimation has been intensively discussed in the literature. In contrast, in this paper we are interested in estimating the value $\ell(\sol)$ of a continuous linear functional $\ell:\cF\to\mathbb R$. Important examples discussed in this paper are the weighted average derivative or the point evaluation functional which are both continuous under appropriate conditions on $\cF$. We establish a lower bound of the maximal mean squared error for estimating $\ell(\sol)$ over a wide range of classes $\cF$ and functionals $\ell$. As a step towards adaptive estimation, we propose in this paper a plug-in estimator of $\ell(\sol)$ which is consistent and minimax optimal. This estimator is based on a linear Galerkin approach which involves the choice of a dimension parameter. We present a method for choosing this parameter in a data driven way combining model selection and Lepski's method. Moreover, it is shown that
the adaptive estimator can attain the minimax optimal rate of convergence within a logarithmic factor.

 Model (\ref{model:NP}--\ref{model:NP2}) has been introduced first by \cite{F03eswc} and \cite{NP03econometrica}, while its identification has been studied  e.g. in \cite{CFR06handbook}, \cite{DFR02} and \cite{FJVB07}. It is interesting to note that recent applications and extensions of this approach include nonparametric tests of exogeneity (\cite{BH07res}), quantile regression models (\cite{HL07econometrica}), or semiparametric modeling (\cite{FJVB05}) to name but a few. 
For example, 
\cite{AC03econometrica},  \cite{BCK07econometrica}, \cite{ChenReiss2008} or \cite{NP03econometrica}  consider sieve minimum distance estimators of $\sol$, while  \cite{DFR02}, \cite{HallHorowitz2005}, \cite{GS06}
or \cite{FJVB07} study penalized least squares estimators.
A linear Galerkin approach to construct an estimator of $\sol$ coming from the inverse problem community (c.f. \cite{EfromovichKoltchinskii2001} or \cite{HoffmannReiss04}) has been proposed by \cite{Johannes2009}.
But  estimating  the
structural function $\sol$ as a whole involves the inversion of the conditional expectation operator of $Z$ given $W$ and  
 generally leads to an ill-posed inverse problem (c.f. \cite{NP03econometrica} or \cite{F03eswc}). This essentially implies  that all
proposed estimators   have under reasonable assumptions  very poor rates of convergence.
In contrast, it might be possible to estimate certain local features of $\sol$, such as the value of  certain linear functionals
at the usual parametric rate of convergence.

The nonparametric estimation of the value of a linear functional  from Gaussian white noise observations is a subject of considerable literature (c.f. \cite{Speckman1979}, \cite{Li1982} or \cite{IbragimovHasminskii1984} in case of  direct observations, while in case of indirect observations we refer to \cite{DonohoLow1992}, \cite{Donoho1994} or \cite{GoldPere2000}). However, nonparametric instrumental regression is in general not a  Gaussian white noise model.
On the other hand, in the former setting the parametric estimation of linear functionals has been addressed in recent years in the econometrics literature. To be more precise, under restrictive conditions on the linear functional $\ell$ and the joint distribution of $(Z,W)$ it is shown in \cite{AC07}, \cite{Santos11}, and \cite{TS07} that it is possible to construct $n^{1/2}$-consistent estimators of $\ell(\sol)$. In this situation, efficiency bounds are derived by \cite{AC07} and, when $\sol$ is not necessarily identified, by \cite{TS07}. We show below, however, that $n^{1/2}$-consistency is not possible for a wide range of linear functionals $\ell$ and joint distributions of $(Z,W)$.

In this paper we establish a minimax theory for the nonparametric estimation of the value of a linear functional $\ell(\sol)$ of the structural function $\sol$.
For this purpose we consider a plug-in estimator $\hell_m:=\ell(\hsol_m)$ of $\ell(\sol)$, where the estimator $\hsol_m$ is proposed by \cite{Johannes2009} and the integer $m$ denotes a dimension to be chosen appropriately.
The accuracy of $\hell_m$ is measured by its maximal mean squared error  uniformly  over the classes $\cF$ and $\cP$, where $\cP$ captures conditions  on the unknown joint distribution $P_{UZW}$ of the random vector $(U,Z,W)$, i.e., $P_{UZW}\in\cP$. The class $\cF$ reflects prior information on the structural function $\sol$, e.g., its level of smoothness, and will be constructed flexible enough to characterize, in particular, differentiable or analytic functions. On the other hand, the condition $P_{UZW}\in\cP$ specifies amongst others some mapping properties of the conditional expectation operator of $Z$ given $W$ implying a certain  decay of its singular values. The construction of $\cP$ allows us to discuss both a polynomial and an exponential decay of those singular values. 
Considering the maximal mean squared error over $\cF$ and $\cP$ we derive a lower bound for estimating $\ell(\sol)$. Given an optimal choice $\mstar$ of the dimension we show that the lower bound is attained by $\hell_{\mstar}$  up to a constant $C>0$, i.e., 
\begin{equation*}
 \sup_{P_{UZW}\in\cP}\sup_{\sol\in\cF}\Ex|\hell_{\mstar}-\ell_\rep(\sol)|^2\leq C\inf_{\breve\ell}\sup_{P_{UZW}\in\cP}\sup_{\sol\in\cF}\Ex|\breve\ell-\ell_\rep(\sol)|^2
\end{equation*}
where the infimum on the right hand side runs over all possible estimators $\breve\ell$. Thereby, the estimator $\hell_{\mstar}$ is minimax optimal even though the optimal choice $\mstar$ depends on the classes $\cF$ and $\cP$, which are unknown in practice.

The main issue addressed in this paper is the construction of a data driven selection method for the dimension parameter which adapts to the unknown classes $\cF$ and $\cP$. When estimating the structural function $\sol$ as a  whole \cite{LoubesMarteau2009} and \cite{Johannes2009} propose adaptive estimators under the condition that the eigenfunctions of the unknown  conditional expectation operator are a-priori given. In contrast our method does not involve this a-priori knowledge and moreover, allows for both 
a polynomial and an exponential decay of the associated singular values. The methodology combines a model selection approach (cf. \cite{BarronBirgeMassart1999} and its detailed discussion in \cite{Massart2007}) and Lepski's method (cf. \cite{Lepski90}). It is inspired by the recent work of \cite{GoldLepski10}. 
To be more precise, the adaptive choice $\hm$  is defined as the minimizer of a random penalized contrast 
criterion\footnote{For a sequence $(a_m)_{m\geq 1}$ having a minimal value in $A\subset \mathbb N$ set $\argmin_{m\in A}\{a_m\}:=\min\{m: a_m\leq a_{m'}\forall m'\in A\}$.}, i.e., 
\begin{subequations}
\begin{equation}\label{full:def:model}
\hm:=\argmin_{1\leq m\leq \hM}\set{\hcontr_m+\hpen_m}
 \end{equation}
 with random integer $\hM$  and random penalty sequence $\hpen:=(\hpen_m)_{m\geq 1}$  to be  defined below and the sequence of contrast $\hcontr:=(\hcontr_m)_{m\geq 1}$ given by
\begin{equation}\label{full:def:contrast}
  \hcontr_m:=\max_{m\leq m'\leq \hM}\set{\absV{\hell_{m'}-\hell_{m}}^2 -\hpen_{m'}}.
\end{equation}
\end{subequations}
 With this adaptive choice $\hm$ at hand the estimator $\hell_{\hm}$ is shown to be minimax optimal within a logarithmic factor over a wide range of classes $\cF$ and $\cP$.

The paper is organized as follows. In Section \ref{sec:lower} we introduce  our basic model assumptions and derive a lower bound 
for
estimating the value of a linear functional in nonparametric instrumental regression. In Section \ref{sec:gen}
we show   consistency  of the proposed estimator first and second that it attains the lower bound 
up to a constant.  We
illustrate the general results by considering classical smoothness assumptions. 
The applicability of these results is demonstrated by various examples such as the estimation of the structural function at a point, of its average or of its weighted average derivative. Finally, in section \ref{sec:adaptive} we construct the random upper bound $\hM$ and the random penalty sequence $\hpen$ used in (\ref{full:def:model}--\ref{full:def:contrast}) to define the data driven selection procedure for the dimension parameter $m$. The proposed adaptive estimator is shown to attain the lower bound within a logarithmic factor.
All proofs can be found in the appendix.

 \section{Complexity of functional estimation: a lower bound.}\label{sec:lower}
\subsection{Notations and basic model assumptions.}
The nonparametric instrumental regression model (\ref{model:NP}--\ref{model:NP2}) leads to a Fredholm equation of the first kind. To be more precise, let us introduce the conditional expectation operator $T\phi:=\Ex[\phi(Z)|W]$ mapping $L^2_Z=\{\phi:\,\Ex[\phi^2(Z)]<\infty\}$ to $L^2_W=\{\psi:\, \Ex[\psi^2(W)]<\infty\}$ (which are endowed with the usual inner products $\skalar_Z$ and
$\skalar_W$, respectively). Consequently, model (\ref{model:NP}--\ref{model:NP2}) can be written as
	\begin{equation}
	\label{model:EE}
	g= T\sol
	\end{equation}
        where the function $g:=\Ex[Y | W]$ belongs to $L^2_W$.  In what follows we always assume that there exists a unique solution $\sol\in L^2_Z$ of equation (\ref{model:EE}), i.e., $g$
        belongs to the range of $\op$, and that the null space of $\op$ is trivial (c.f. \cite{EHN00} or \cite{CFR06handbook} in the special case of nonparametric instrumental
        regression). Estimation of the structural function $\sol$ is thus linked with the inversion of the operator $\op$. Moreover, we suppose throughout the paper that $T$ is
        compact which is under fairly mild assumptions satisfied (c.f. \cite{CFR06handbook}). Consequently, a
        continuous generalized inverse of $\op$ does not exist as long as the range of the operator $\op$ is an infinite dimensional subspace of $L^2_W$. This corresponds to the
        setup of ill-posed inverse problems.

In this section we show that the obtainable accuracy of any estimator of the value $\ell(\sol)$ of a linear functional can be essentially determined by regularity
conditions imposed  on the structural  function $\sol$ and the conditional expectation operator $\op$.  In this paper these conditions are characterized by different weighted norms in $L^2_Z$
with respect to a pre-specified orthonormal basis $\{\basZ_j\}_{j\geq 1}$   in $L^2_Z$, which we formalize  now.
 Given  a positive sequence of  weights $w:=(w_j)_{j\geq1}$ we define the
weighted norm $\normV{\phi}_w^2:=\sum_{j\geq 1}w_j|\skalarV{\phi,\basZ_j}_Z|^2$, $\phi\in L^2_Z$, the completion  $\cF_w$ of $L^2_Z$ with 
respect to $\norm_w$ and the ellipsoid $\cF_{w}^r := \big\{\phi\in \cF_w: \normV{\phi}_{w}^2\leq r\big\}$ with radius $r>0$.  We shall stress that the basis 
$\{\basZ_j\}_{j\geq 1}$ does not necessarily correspond to the eigenfunctions of $\op$. In the following we write $a_n\lesssim b_n$  when there exists a generic constant $C>0$ such that  $a_n\leqslant C\, b_n$  for sufficiently large $ n\in\N$ and  $a_n\sim b_n$ when $a_n\lesssim b_n$ and $b_n\lesssim a_n$ simultaneously.

 \paragraph{Minimal regularity conditions.}
Given a nondecreasing sequence of weights $\sw:=(\sw_j)_{j\geqslant 1}$, we suppose, here and
subsequently, that the structural function $\sol$ belongs to the ellipsoid $\cF_{\sw}^\sr$
for some $\sr>0$. The ellipsoid $\cF_\sw^\sr$ captures all the prior information (such as smoothness) about the unknown
structural function $\sol$. 
Observe that the dual space of $\cF_\sw$ can be identified with $\cF_{1/\sw}$ where $1/\sw:=(\sw_j^{-1})_{j\geq 1}$ (cf. \cite{KreinPetunin1966}). 
To be more precise, for all $\phi\in\cF_{\sw}$ the value $\skalarV{\rep,\phi}_Z$ is well defined for all $\rep\in\cF_{1/\sw}$ and by Riesz's Theorem there exists a unique $\rep\in\cF_{1/\sw}$ such that $\ell(\phi)=\skalarV{\rep,\phi}_Z=:\ell_{\rep}(\phi)$.
In certain applications one might not only be interested  in the performance of an estimation procedure of  $\ell_\rep(\sol)$ for a given representer $\rep$, 
but also for $\rep$ varying over the ellipsoid $\cF_\hw^\hr$ with 
radius  $\hr>0$ for a nonnegative sequence  $\hw:=(\hw_j)_{j\geq1}$  
satisfying  $\inf_{j\geq 1}\{\hw_j\bw_j\}>0$.  Obviously, $\cF_\hw$ is a subset of $\cF_{1/\bw}$.

Furthermore, as usual in the context of ill-posed inverse problems, we specify some mapping properties of
the  operator under consideration.
Denote by $\cT$ the set of all compact operators mapping $L^2_Z$ into $L^2_W$.  Given a sequence of weights $\Opw:=(\Opw_j)_{j\geqslant 1}$ and $\Opd\geqslant
1$ we define the subset $\Opwd$ of $\cT$ by
\begin{equation}\label{bm:link}
\Opwd:=\Bigl\{ T\in\cT:\quad   \normV{\phi}_{\Opw}^2/d\leqslant \normV{\Op \phi}_W^2\leqslant {\Opd}\, \normV{\phi}_{\tw}^2,\quad \forall \phi \in L^2_Z\Bigr\}.
\end{equation}
Notice first that any operator $T\in\Opwd$ is injective if the sequence $\opw$ is strictly positive.
Furthermore, for all $T\in\Opwd$ it follows that  $\opw_j/d\leq \normV{T\basZ_j}_W^2\leq d\Opw_j$ for all $j\geq 1$ and if $(s_j)_{j\geq1}$ denotes the ordered  sequence of singular values of $T$ then it holds $\opw_j/d\leq s^2_j\leq d \tw_j$.
In other words, the sequence $\Opw$ specifies the decay of the singular values of $\Op$.
In what follows, all the results are derived under regularity conditions on the structural function $\sol$ and the conditional expectation operator $\Op$ described through the
sequence $\sw$ and $\Opw$, respectively.  We provide  illustrations of these conditions below by assuming a \lq\lq regular decay\rq\rq\ of these sequences.  The next
assumption summarizes our minimal regularity conditions on these sequences.
\begin{assA}\label{ass:reg} Let  $\sw:=(\sw_j)_{j\geqslant 1}$, $\rw:=(\rw_j)_{j\geqslant 1}$ and  $\Opw:=(\Opw_j)_{j\geqslant 1}$ be strictly positive sequences of weights
  with   $\sw_1=\rw_1=\Opw_1= 1$ such that   $\sw$ is nondecreasing with $|j|^3\sw_j^{-1}=o(1)$ as $j\to\infty$, $\rw$  satisfies $\inf_{j\geq1}\{\rw_j\sw_j\}>0$ and $\Opw$ is a
  nonincreasing sequence with $\lim_{j\to\infty}v_j=0$.
\end{assA}
\begin{rem}\label{rem:ill:cases}
We illustrate Assumption \ref{ass:reg} for typical choices of $\sw$ and $\opw$ usually studied in the literature (c.f. \cite{HallHorowitz2005}, \cite{ChenReiss2008} or \cite{JoVBVa07}),
 that is,
 \begin{enumerate}
 \item[(pp)] $\sw_j  \sim |j|^{2p}$ with  $p>3/2$, $\tw_j  \sim |j|^{-2a}$,  $a>0$, and
\begin{enumerate}
\item[(i)]$[h]_j^2\sim |j|^{-2s}$, $s>1/2-p$ or
\item[(ii)]$\rw_j  \sim |j|^{2s}$, $s>-p$.
\end{enumerate}
\item[(pe)]$\sw_j  \sim |j|^{2p}$, $p>3/2$ and $\tw_j  \sim \exp(-|j|^{2a})$, $a>0$, and 
\begin{enumerate}
\item[(i)]$[h]_j^2\sim |j|^{-2s}$, $s>1/2-p$ or
\item[(ii)]$\rw_j  \sim |j|^{2s}$, $s>-p$.
\end{enumerate}
\item[(ep)]$\sw_j  \sim \exp(|j|^{2p})$, $p>0$   and $\tw_j  \sim |j|^{-2a}$,.$a>0$, and
\begin{enumerate}
\item[(i)]$[h]_j^2\sim |j|^{-2s}$, $s\in\mathbb R$ or
\item[(ii)]$\rw_j  \sim |j|^{2s}$, $s\in\mathbb R$.
\end{enumerate}
\end{enumerate}
Note that condition $|j|^3\sw_j^{-1}=o(1)$ as $j\to\infty$ is automatically satisfied in case of (ep). In the other two cases this condition states under classical smoothness assumptions  that, roughly speaking, the structural function $\sol$ has to be differentiable.
\hfill$\square$\end{rem}

In order to formulate   below      the lower  as well as the upper bound let us define for $x\geq 1$%
\begin{equation}\label{m:known:operator}\mstarx:=\argmin \limits_{m \in \N} \set{\frac{\max \set{\frac{\Opw_m}{\sw_m},x^{-1}}}{\min \set{\frac{\Opw_m}{\sw_m},x^{-1}}}}, \quad
  \astarx:=\max \set{\frac{\Opw_{\mstarx}}{\sw_{\mstarx}},x^{-1}}.\end{equation}
We shall see that the minimax optimal rate  is determined   by the sequence $\drep:=(\drepn)_{n\geq1}$, in case of a fixed representer $\rep$,  and
$\drw:=(\drwn)_{n\geq1}$ in case of a representer varying over the class $\cF_\rw^\rr$. These sequences are given  for all $x\geq1$  by
\begin{equation}\label{rate:known:operator}
\drepx:=\max \set{\astarx \sum\limits_{j=1}^{\mstarx} \frac{[\rep]_j^2}{\Opw_j}, \sum_{j>\mstarx}\frac{[\rep]_j^2}{\sw_j}}\quad\mbox{and}\quad\drwx:=\astarx\,\max \limits_{1\leq j \leq {\mstarx}}\set{ \frac{1}{\rw_j \Opw_j}}.
\end{equation}
In case of adaptive estimation the rate of convergence is given by $\drepad:=(\drepnl)_{n\geq1}$  and
$\drwad:=(\drwnl)_{n\geq1}$, respectively.
For ease of notation let $\md:=\mstarnl$ and $\ad:=\astarnl$. The  bounds  established below need the following additional assumption, which is satisfied in all cases   used to illustrate the results.
\begin{assA}\label{ass:reg:II}
Let  $\sw$ and $\Opw$ be  sequences such that%
\begin{equation}\label{eta:known:operator}
0<\kappa:=\inf_{n\geq1} \set{(\astarn)^{-1} \min  \set{\frac{\Opw_{\mstarn}}{\sw_{\mstarn}},n^{-1}} }\leq1.\end{equation}
\end{assA}

\subsection{Lower bounds.}
The results derived below involve assumptions on the conditional moments of the random variables $U$ given $W$, captured by $\cU_\sigma$, which contains all conditional distributions of $U$ given $W$, denoted by $P_{U|W}$, satisfying $\Ex[U|W]=0$ and $\Ex[U^4|W]\leq \sigma^4$ for some $\sigma>0$. 
The next assertion gives a lower bound for the mean squared error of any estimator when estimating the value $\ell_\rep(\sol)$ of a linear functional with given representer $\rep$
and structural function $\sol$ in  the function class $\cF_\sw^\sr$.

\begin{theo}\label{res:lower}Assume an iid. $n$-sample of $(Y,Z,W)$ from the model  (\ref{model:NP}--\ref{model:NP2}).
Let $\sw$ and $\Opw$ be sequences satisfying Assumptions \ref{ass:reg} and \ref{ass:reg:II}.
Suppose that $\sup_{j\geq 1} \Ex [\basZ_j^4(Z)|W]\leq \eta^4$,  $\eta\geq1$, and $\sigma^4\geq \big(\sqrt 3+4\sr \,\eta^2\sum_{j\geq 1}\sw_j^{-1}\big)^2$.  Then for all $n\geq 1$ we have
\begin{equation*}   \inf_{\breve\ell}\inf_{\op\in\Opwd} \sup_{P_{U|W}\in \cU_\sigma} \sup_{\sol \in\cF_\sw^{\sr}}  \Ex|\breve{\ell}-\ell_\rep(\sol)|^2\geqslant
  \frac{\kappa}{4} \,\min \bigg( \frac{1}{2\,\Opd}\,,\,  \sr\bigg)\,\drepn
\end{equation*}
where the first infimum runs over all possible estimators $\breve\ell$.
\end{theo}
\begin{rem}
In the proof of the lower bound we consider a test problem based on two hypothetical structural functions. For each test function the condition $\sigma^4\geq\big(\sqrt 3+4\sr \,\eta^2\sum_{j\geq 1}\sw_j^{-1}\big)^2$ ensures a certain complexity of the hypothetical model in a sense that it allows for Gaussian residuals.
  This specific case is only needed to simplify the calculation of the distance between distributions corresponding to
  different structural functions. A similar assumption has been used by \cite{ChenReiss2008} in order to derive a lower bound
for the estimation of the structural function $\sol$ itself. In particular, the authors show that in opposite to the present work an one-dimensional subproblem is not sufficient to
describe the full difficulty in estimating $\sol$.

On the other hand, below we derive an upper bound assuming that $P_{U|W}$ belongs to $\cU_\sigma$ and that the joint distribution of
  $(Z,W)$ fulfills in addition Assumption \ref{ass:A1}. Obviously in this situation Theorem \ref{res:lower} provides a lower bound for any estimator as long as $\sigma$ is sufficiently
  large. 
\hfill$\square$\end{rem}

\begin{rem}
The regularity conditions imposed on the structural function $\sol$ and the conditional expectation operator $T$ involve only the
  basis $\{\basZ_j\}_{j\geq1}$ in $L^2_Z$. Therefore, the lower bound derived in Theorem \ref{res:lower} does not capture the influence of the basis $\{\basW_l\}_{l\geq1}$ in
  $L^2_W$ used below to construct an  estimator  of the value $\ell_\rep(\sol)$. In other words, this estimator attains the lower bound only if $\{\basW_l\}_{l\geq1}$ is
  chosen appropriately. 
\hfill$\square$\end{rem}

\begin{rem}
The rate $\drep$ of the lower bound is never faster than the parametric rate, that is, $\drepn\geq n^{-1}$.
 Moreover, it is easily seen that the lower bound rate is
  parametric if and only if
$\sum_{j\geq1}[\rep]_j^2\Opw_j^{-1}<\infty$.
 This condition does not involve the
  sequence $\sw$ and hence, attaining a parametric rate is independent of the regularity conditions we are willing to impose on the structural function.
  \hfill$\square$\end{rem}

 The following assertion gives a lower bound over the ellipsoid $\cF_\rw^{\rr}$ of representer. Consider the function $\rep^*:=\rr\rw_{j^*}^{-1/2}e_{j^*}$ with $j^*:={\argmax}_{1\leq j\leq \mstarn}\{ (\rw_j\Opw_j)^{-1} \}$ which obviously belongs to $\cF_\rw^{\rr}$. Corollary \ref{res:coro:lower} follows then by calculating the value of the lower bound in Theorem \ref{res:lower} for the specific representer $\rep^*$ and, hence we omit its proof.

\begin{coro}\label{res:coro:lower} Let  the assumptions of Theorem \ref{res:lower} be satisfied. Then for all $n\geq 1$  we have 
\begin{equation*}    \inf_{\breve\ell}\inf_{\op\in\Opwd}\sup_{P_{U|W}\in \cU_\sigma} \sup_{\sol \in\cF_\sw^{\sr},\,\rep \in\cF_\rw^{\rr}}  \Ex|\breve{\ell}-\ell_\rep(\sol)|^2\geqslant
  \frac{\tau\kappa}{4} \,\min \bigg( \frac{1}{2\,\Opd}\,,\,  \sr\bigg)\,\drwn
\end{equation*}
where the first infimum runs over all possible estimators $\breve\ell$.
\end{coro}

\begin{rem}
If the lower bound given in  Corollary \ref{res:coro:lower} tends to zero then $(\rw_{j}\sw_{j})_{j\geq 1}$ is a divergent sequence. In other words, without any additional restriction on $\sol$, that is, $\sw\equiv 1$,  consistency of an estimator of $\ell_\rep(\sol)$ uniformly over all $\sol\in \cF_\sw^\sr$ and all $\rep\in\cF_\rw^\rr$ is only possible under restrictions
  on the representer $\rep$ in the sense that $\rw$ has to be a divergent sequence. This obviously reflects the ill-posedness of the underlying inverse problem. 
 \hfill$\square$\end{rem}

 \section{Minimax optimal estimation.}\label{sec:gen}
\subsection{Estimation by dimension reduction and thresholding.}
In addition to the basis $\{e_j\}_{j\geq 1}$ in  $L_Z^2$ used to establish the lower bound  we consider now also a second basis $\{f_l\}_{l\geq 1}$ in $L_W^2$. 
\paragraph{Matrix and operator notations.} Given $m\geq 1$, $\cE_m$ and $\cF_m$ denote the subspace of $L^2_Z$ and $L^2_W$ spanned by the functions $\{e_j\}_{j=1}^m$ and
$\{f_l\}_{l=1}^m$ respectively. $E_m$ and $E_m^\perp$ (resp. $F_m$ and $F_m^\perp$) denote the orthogonal projections on $\cE_m$ (resp. $\cF_m$) and its orthogonal complement
$\cE_m^\perp$ (resp. $\cF_m^\perp$), respectively. Given an operator $K$ from $L^2_Z$ to $L^2_W$ we denote its inverse by $K^{-1}$ and its adjoint by
$K^*$.   
If we restrict $F_m K E_m$ to an operator from $\cE_m$ to $\cF_m$, then it can be represented by a matrix $[K]_{\um}$ with entries $[K]_{l,j}=\skalarV{Ke_j,f_l}_W$ for $1\leq j,l\leq m$. Its spectral norm is denoted by  $\normV{[K]_{\um}}$, its inverse   by $[K]_{\um}^{-1}$ and its transposed by $[K]_{\um}^t$. We write $I$ for the identity operator
and $\Diag_\opw$ for the diagonal operator with singular value decomposition $\{v_j,\basZ_j,\basW_j\}_{j\geq1}$.  Respectively, given functions $\phi\in L^2_Z$ and  $\psi\in L^2_W$ we define by $[\phi]_{\um}$ and $[\psi]_{\um}$ $m$-dimensional vectors with entries 
$[\phi]_{j}=\skalarV{\phi,e_j}_Z$ and $[\psi]_{l}=\skalarV{\psi,f_l}_W$ for $1\leq j,l\leq m$.

Consider the conditional expectation operator $T$ associated with $(Z,W)$. If $[\basZ(Z)]_\um$ and $[\basW(W)]_\um$ denote random vectors with entries $\basZ_j(Z)$ and $\basW_j(W)$, $1\leq j\leq m$, respectively, then it holds $[T]_{\um}=\Ex [\basW(W)]_{\um}[\basZ(Z)]_{\um}^t$. Throughout the paper $[T]_{\um}$ is assumed  to be nonsingular for all $m\geq 1$, so that $[T]_{\um}^{-1}$ always exists. Note that it is a nontrivial problem
to determine when such an assumption holds (cf. \cite{EfromovichKoltchinskii2001} and references therein).

\paragraph{Definition of the estimator.}Let $(Y_1,Z_1,W_1),\dotsc,(Y_n,Z_n,W_n)$ be an iid. sample of $(Y,Z,W)$. 
Since $[T]_{\um}=\Ex [\basW(W)]_{\um}[\basZ(Z)]_{\um}^t$ and $[g]_{\um}=\Ex{Y[\basW(W)]_{\um}}$ we construct estimators by using their empirical counterparts, that is,
\begin{equation*}
[\widehat{T}]_{\um}:= \frac{1}{n}\sum_{i=1}^n [\basW(W_i)]_{\um}[\basZ(Z_i)]_{\um}^t  \quad\mbox{ and }\quad[\widehat{g}]_{\um}:= \frac{1}{n}\sum_{i=1}^n Y_i[\basW(W_i)]_{\um}.
\end{equation*}
Then the  estimator of the linear functional $\ell_\rep(\sol)$ is defined for all $m\geq 1$ by
\begin{equation}\label{gen:def:est}
\hell_m:=
\left\{\begin{array}{lcl} 
[\rep]_{\um}^t[\hop]_{\um}^{-1} [\widehat{g}]_{\um}, && \mbox{if $[\hop]_{\um}$ is nonsingular and }\normV{[\hop]^{-1}_{\um}}\leq \sqrt{n},\\
0,&&\mbox{otherwise}.
\end{array}\right.
\end{equation}
In fact, the estimator $\hell_m$ is obtained from the linear functional $\ell_\rep(\sol)$ by replacing the unknown structural function $\sol$ by an estimator proposed by \cite{Johannes2009}. 
\begin{rem}
If $Z$ is continuously distributed one might be also interested in estimating the value $\int_{\cZ} \sol(z)\rep(z)dz$ where $\cZ$ is the support of $Z$. Assume that this integral and also $\int_{\cZ} \rep(z) e_j(z)dz$ for $1\leq j\leq m$ are well defined. Then we can cover the problem of estimating $\int_{\cZ} \sol(z)\rep(z)dz$ by simply replacing $[\rep]_{\um}$ in the definition of $\hell_m$ by a $m$-dimensional vector with entries $\int_{\cZ} \rep(z) e_j(z)dz$ for $1\leq j\leq m$. Hence for  $\int_{\cZ} \sol(z)\rep(z)dz$ the results below follow mutatis mutandis. \hfill$\square$
\end{rem}

\paragraph{Moment assumptions.}
 
Besides the link condition \eqref{bm:link} for the conditional expectation operator $T$  we need moment conditions  on the basis, more specific, on  the  random variables $e_j(Z
)$ and $f_l(W)$ for $j,l\geq 1$, which we summarize in the next assumption. 
\begin{assA}\label{ass:A1}
There exists $\eta\geqslant 1$ such that the joint distribution of $(Z,W)$ satisfies  
\begin{itemize}
\item[(i)] $\sup_{j\in \N}\Ex [e_j^2(Z)|W]\leqslant \eta^2$ and $\sup_{l\in \N}\Ex [f_l^4(W)]\leqslant \eta^4$;
\item[(ii)] $\sup_{j,l\in\N} \Ex| e_j(Z)f_l(W)- \Ex [e_j(Z)f_l(W)]|^k\leqslant \eta^{k} k!$, $k=3,4,\dotsc$.
\end{itemize}
\end{assA}
Note that condition $(ii)$  is also known as Cramer's condition, which is sufficient to obtain an exponential bound for large deviations of the centered random variable $e_j(Z)f_l(W)- \Ex [e_j(Z)f_l(W)]$ (c.f. \cite{Bosq1998}). Moreover, any joint distribution  of $(Z,W)$ satisfies Assumption \ref{ass:A1} for sufficiently large $\eta$   if  the basis $\{e_j\}_{j\geq1}$ and $\{f_l\}_{l\geq1}$ are uniformly bounded, which holds, for example, for the trigonometric basis considered in Subsection \ref{sec:sob}.
 
\subsection{Consistency.}
 The next assertion summarizes  sufficient conditions to ensure consistency of the estimator ${\hell_m}$ introduced  in (\ref{gen:def:est}). Let $\sol_m\in\mathcal E_m$ with $[\sol_m]_{\um}=[\op]_{\um}^{-1}[g]_{\um}$. Obviously, up to the threshold, the estimator $\hell_m$ is the empirical counterpart of $\ell_\rep(\sol_m)$.
In Proposition \ref{res:gen:prop:cons} consistency of the estimator $\hell_m$ is only obtained under the condition 
\begin{equation}\label{eq:cons:solm}
 \|\sol-\sol_m\|_\sw=o(1) \text{ as }m\to\infty
\end{equation}
which does not hold true in general. Obviously \eqref{eq:cons:solm} implies the convergence of $\ell_h(\sol_m)$ to $\ell_h(\sol)$ as $m$ tends to infinity for all $\rep\in\cF_{1/\sw}$.
  \begin{prop}\label{res:gen:prop:cons} Assume an iid. $n$-sample of $(Y,Z,W)$ from the model (\ref{model:NP}--\ref{model:NP2})  with  $P_{U|W} \in \cU_\sigma$ and   joint
  distribution of $(Z,W)$ fulfilling Assumption \ref{ass:A1}. Let the dimension  parameter $m_n$ satisfy $m_n^{-1}=o(1)$, $m_n=o(n)$,
  \begin{equation}\label{eq:5}
\normV{[\rep]_{\umn}^t [T]_{\umn}^{-1}}^2=o(n),\mbox{ and }
 m_n^3\normV{[\Op]_{\umn}^{-1}}^2=O(n)\text{ as }n\to\infty.
  \end{equation}
 If \eqref{eq:cons:solm} holds true  then 
 $\Ex|\hell_{m_n}-{\ell_\rep(\sol)}|^2=o(1)$ as  $n\to\infty$ for all $\sol\in\cF_\sw$ and $\rep\in\cF_{1/\sw}$. 
  \end{prop}
 Notice that  condition \eqref{eq:cons:solm} also involves the basis $\{\basW_l\}_{l\geq1}$ in $L^2_W$. In what follows, we introduce an alternative but stronger
  condition to guarantee \eqref{eq:cons:solm} which extends the link condition \eqref{bm:link}.  We denote by $\cTdDw$ for some $\tD\geq \td$ the subset of
  $\cTdw $ given by
\begin{equation}\label{bm:link:gen}
\cTdDw:=\Bigl\{ T\in \cTdw :\quad \sup_{m\in\N}\normV{[\Diag_{\tw}]^{1/2}_{\um}[T]^{-1}_{\um}}^2\leqslant \tD\Bigr\}.
\end{equation}
\begin{rem}\label{rem:cons}
 If  $T\in\cTdw $ and if in addition its singular value decomposition is given by $\{s_j,e_j, f_j\}_{j\geq1}$
then for all $m\geq 1$ the matrix $[\op]_{\um}$ is diagonalized with diagonal entries $[\op]_{j,j}=s_j$, $1\leq j\leq m$.
  In this situation it is easily seen that $\sup_{m\in\N}\normV{[\Diag_{\tw}]_{\um}^{1/2}[T]^{-1}_{\um}}^2\leqslant \td$ and,
  hence $T$ satisfies the extended link condition \eqref{bm:link:gen}, that is, $T\in \cTdDw$. Furthermore, it holds $\cTdw =\cTdDw$ for suitable $\tD>0$, if $T$ is a small perturbation of $\Diag_{\tw}^{1/2}$ or if $T$ is strictly positive (c.f.
  \cite{EfromovichKoltchinskii2001} or \cite{CardotJohannes2008}, respectively). \hfill$\square$
\end{rem}

\begin{rem}\label{rem:cons:2}
 Once both basis $\{\basZ_j\}_{j\geq1}$ and $\{\basW_l\}_{l\geq1}$ are specified the extended link condition \eqref{bm:link:gen} restricts the class of joint
  distributions of $(Z,W)$  such that \eqref{eq:cons:solm} holds true.
Moreover, under  \eqref{bm:link:gen} the estimator $\hsol_m$ of $\sol$ proposed by
\cite{Johannes2009} can attain the minimax optimal rate. In this
  sense, given a joint distribution of $(Z,W)$ a basis $\{\basW_l\}_{l\geq1}$ satisfying condition \eqref{bm:link:gen} can be interpreted as optimal instruments (c.f. \cite{N90econometrica}).\hfill$\square$
\end{rem}
  
  \begin{rem}
  For each
  pre-specified basis $\{\basZ_j\}_{j\geq1}$ we can theoretically construct a basis $\{\basW_l\}_{l\geq1}$ such that \eqref{bm:link:gen} is equivalent to 
  the link condition \eqref{bm:link}. To be more precise, if $T\in\cTdw $, which involves only the basis $\{\basZ_j\}_{j\geq1}$, then the fundamental
  inequality of \cite{Heinz51} implies $ \normV{ (T^*T)^{-1/2} \basZ_j}_Z^2\leq d \tw_j^{-1}$. Thereby, the function $(T^*T)^{-1/2} \basZ_j$ is an element of
  $L^2_Z$ and, hence $\basW_j:= T (T^*T)^{-1/2} \basZ_j$, $j\geq1$, belongs to $L^2_W$. Then it is easily checked that $\{\basW_l\}_{l\geq1}$ is
  a basis of the closure of the range of $T$ which may be completed to a basis of $L^2_W$.
  Obviously $[T]_{\um}$ is symmetric and moreover, strictly positive since $\skalarV{ \op\basZ_j,\basW_l}_W=\skalarV{ (\op^*\op)^{1/2}\basZ_j, \basZ_l}_Z $ for all $j,l\geq 1$. Thereby, we can apply Lemma A.3 in \cite{CardotJohannes2008}
  which gives $\cTdw= \cTdDw$ for sufficiently large $D$. We are currently exploring the  data driven choice of the basis  $\{\basW_l\}_{l\geq1}$.
 \hfill$\square$\end{rem}

Under the extended link condition \eqref{bm:link:gen}  the next assertion summarizes sufficient conditions to ensure consistency. 
 \begin{coro}\label{res:gen:coro:cons} The conclusion of Proposition \ref{res:gen:prop:cons} still holds true without imposing condition \eqref{eq:cons:solm}, if the sequence $\tw$ satisfies Assumption \ref{ass:reg}, the
   conditional expectation operator $\Op$ belongs to $\OpwdD$, and \eqref{eq:5} is substituted by
  \begin{equation}\label{eq:6}
\sum_{j=1}^{m_n}[\rep]_{j}^2\tw_j^{-1}=o(n)\quad\mbox{ and }\quad m_n^3=O(n\opw_{m_n})\quad\mbox{ as }n\to\infty.
  \end{equation}
\end{coro}
\subsection{An upper bound.}
The last assertions show that the estimator $\hell_m$   defined in \eqref{gen:def:est}  is consistent for 
all structural functions  and  representers belonging to $\cF_\bw$ and $\cF_{1/\bw}$, respectively. The following theorem provides now an upper bound if $\sol$ belongs to an ellipsoid $\cF_\bw^\br$.  
In this situation the rate 
$\drep$ of the lower bound given  in Theorem \ref{res:lower} provides up to a constant also an upper bound of 
the  estimator $\hell_{\mstarn}$.  Thus we have proved that  the rate 
$\drep$
 is optimal and, hence  $\hell_{\mstarn}$   is  minimax optimal. 

\begin{theo}\label{res:upper}  Assume an iid. $n$-sample of $(Y,Z,W)$ from the model (\ref{model:NP}--\ref{model:NP2}) with     joint
  distribution of $(Z,W)$ fulfilling Assumption \ref{ass:A1}. Let  Assumptions \ref{ass:reg} and \ref{ass:reg:II} be satisfied.
 Suppose that the  dimension parameter  $\mstarn$  given by
\eqref{m:known:operator} satisfies
  \begin{equation}\label{minimax upper c1}
(\mstarn)^3 \max\set{|\log \drepn|, (\log \mstarn)}=o(\sw_{\mstarn} ),\quad\text{ as } n\to\infty,
\end{equation}
then we have for all $n\geq 1$ 
\begin{equation*}
\sup_{\op\in\OpwdD}\sup_{P_{U|W}\in\cU_\sigma}\sup_{\sol \in \cF_\sw^{\sr}} \Ex|\hell_{\mstarn}- \ell_\rep(\sol)|^2\leq \mathcal C \, \drepn
\end{equation*}
for a constant $\mathcal C>0$ only depending on the classes $\Fswsr$, $\cTdDw$, the constants $\sigma$, $\eta$ and the representer $h$.
\end{theo}

The following assertion states an upper bound uniformly over the class $\cF_\rw^\rr$ of representer. Observe that $\|\rep\|_{1/\sw}^2\leq\rr$ and $\drepn\leq \rr \, \astar\,\max_{1\leq j \leq {\mstar}}\{(\rw_j \Opw_j)^{-1}\}=  \rr \,\drwn$ for all $\rep \in \cF_\rw^\rr$. 
Employing these estimates the proof of the next result follows line by line the proof of Theorem \ref{res:upper} and is thus omitted.
\begin{coro}\label{res:coro:upper} Let  the assumptions of Theorem \ref{res:upper} be satisfied where we substitute condition \eqref{minimax upper c1} by
$(\mstarn)^3 \max\set{|\log \drwn|, (\log \mstarn)}=o(\sw_{\mstarn} )$ as $ n\to\infty$.
Then we have  
\begin{equation*}
 \sup_{\op\in\OpwdD} \sup_{P_{U|W}\in\cU_\sigma} \sup_{\sol \in\cF_\sw^{\sr},\,\rep \in\cF_\rw^{\rr}}  \Ex|\hell_{\mstarn}- \ell_\rep(\sol)|^2 \leq  \mathcal C \, \drwn
  \end{equation*}
for a constant $\mathcal C>0$ only depending on the classes $\Fswsr$, $\Frwrr$, $\cTdDw$ and the constants $\sigma$, $\eta$.
\end{coro}

 \subsection{Illustration by classical smoothness assumptions.}\label{sec:sob}
Let us illustrate our general results by considering classical smoothness assumptions.
To simplify the presentation we follow  \cite{HallHorowitz2005}, and suppose that the marginal distribution of the scalar regressor $Z$ and the scalar instrument $W$ are uniformly distributed on the interval $[0,1]$. All the results below can be easily extended  to the multivariate case. In the univariate case, however,  both  Hilbert spaces $L^2_Z$ and $L^2_W$ equal $L^2[0,1]$.
Moreover, as a basis $\{e_j\}_{j\geq 1}$ in $L^2[0,1]$ we choose the trigonometric basis given by 
\begin{equation*}
e_{1}:\equiv1, \;e_{2j}(t):=\sqrt{2}\cos(2\pi j t),\; e_{2j+1}(t):=\sqrt{2}\sin(2\pi j t),t\in[0,1],\; j\in\N.\end{equation*}
In this subsection also the  second basis $\{f_l\}_{l\geq1}$ is given by the trigonometric basis. In this situation,  the moment conditions formalized in Assumption \ref{ass:A1} are automatically fulfilled.

Recall the typical choices of the sequences $\sw$, $\rw$, and $\opw$ introduced in Remark \ref{rem:ill:cases}. If  $\sw_j\sim |j|^{2p}$, $p>0$, as in case (pp) and (pe), then  
 $\cF_\sw$ coincides with the Sobolev space  of $p$-times differential periodic functions
(c.f. \cite{Neubauer1988,Neubauer88}). 
In case of (ep) it is well known that $\cF_{\sw}$ contains only analytic functions if $p>1$(c.f. \cite{Kawata1972}).
Furthermore, we consider two special cases describing a \lq\lq regular decay\rq\rq\ of the unknown singular values of $\op$.
In case of (pp) and (ep)  we consider a polynomial decay of the sequence $\tw$. Easy calculus shows  that  any operator $T$ satisfying  the link condition \eqref{bm:link} acts like integrating   $(a)$-times and hence is called {\it finitely smoothing} (c.f. \cite{Natterer84}). In case of (pe) we consider an exponential decay of $\opw$ and  it can  easily be seen that  $\op\in\cTdw$ implies $\cR(\op)\subset C^\infty[0,1]$, therefore  the operator $\op$ is called {\it infinitely smoothing} (c.f. \cite{Mair94}). 
In the next assertion we present the order of sequences $\drep$ and  $\drw$ which were shown to be minimax-optimal.

\begin{prop}\label{res:lower:sob}
Assume an iid. $n$-sample of $(Y,Z,W)$ from the model (\ref{model:NP}--\ref{model:NP2}) with $T\in\cTdDw$ and  $P_{U|W}\in\cU_\sigma$. Then for the example configurations of Remark \ref{rem:ill:cases} we obtain
\\[-4ex]
 \begin{enumerate}
 \item[(pp)] $\mstarn \sim n^{1/(2p+2a)}$ and
\begin{enumerate}\item[(i)]
$\drepn\sim 
\left\{\begin{array}{lcl} 
 n^{-(2p+2s-1)/(2p+2a)}, && \mbox{if $s-a<1/2$},\\
 n^{-1}\log n, && \mbox{if $s-a=1/2$},\\
 n^{-1}, && \mbox{otherwise,}
\end{array}\right.$
\item[(ii)]  $\drwn\sim \max(n^{-(p+s)/(p+a)},n^{-1})$.
\end{enumerate}
\item[(pe)] $\mstarn\sim \log(n(\log n)^{-p/a})^{1/(2a)}$ and
\begin{enumerate}\item[(i)]
$\drepn\sim (\log n)^{-(2p+2s-1)/(2a)}$,
\item[(ii)] 
$\drwn\sim(\log n)^{-(p+s)/a}$.
\end{enumerate}
\item[(ep)]  $\mstarn\sim \log(n(\log n)^{-a/p})^{1/(2p)}$
 and
\begin{enumerate}\item[(i)] 
$\drepn\sim
\left\{\begin{array}{lcl} 
 n^{-1}(\log n)^{(2a-2s+1)/(2p)}, && \mbox{if $s-a<1/2$},\\
 n^{-1}\log(\log n), && \mbox{if $s-a=1/2$},\\
 n^{-1}, && \mbox{otherwise,}
\end{array}\right.$
\item[(ii)]  
$\drwn\sim \max(n^{-1}(\log n)^{(a-s)/p},n^{-1})$.
\end{enumerate}
\end{enumerate}
\end{prop}

\begin{rem}\label{rem:upper:sob:1}
 As we see from Proposition \ref{res:lower:sob}, if the value of $a$ increases the obtainable optimal rate of convergence decreases. Therefore, the parameter $a$ is often called {\it degree of ill-posedness} (c.f. \cite{Natterer84}).  On the other hand, an increasing of the value $p$ or $s$ leads to a faster optimal rate. Moreover, in the cases (pp) and (ep) the parametric rate $n^{-1}$ is obtained independent of the smoothness assumption imposed on the structural function $\sol$ (however, $p\geq 3/2$ is needed) if the representer  is smoother
 than the degree of ill-posedness of $\op$, i.e., \textit{(i)} $s\geqslant a-1/2$ and \textit{(ii)} $s\geq a$. 
Moreover, it is easily seen that if $[\rep]_j  \sim \exp(-|j|^{s})$ or $\rw_j  \sim \exp(|j|^{2s})$, $s>0$, then the minimax convergence rates are always parametric for any polynomial sequences $\sw$ and $\opw$.
\hfill$\square$\end{rem}

\begin{example}\label{exp:evaluation}
Suppose we are interested in estimating the value $\sol(t_0)$ of the structural function $\sol$ evaluated at a point $t_0\in[0,1]$. Consider the representer given by $ \rep_{t_0}=\sum_{j=1}^\infty e_j(t_0)e_j$. Let $\sol\in\cF_\sw$. Since $\sum_{j\geq1}\bw_j^{-1}<\infty$ (cf. Assumption~\ref{ass:reg}) it holds $\rep\in\cF_{1/\sw}$
and hence the point evaluation functional in $t_0\in[0,1]$, i.e.,  $\ell_{\rep_{t_0}}(\sol)=\sol(t_0)$, is well defined. 
In this case, the estimator $ \hell_m$  introduced in \eqref{gen:def:est} writes for all $m\geq 1$ as
\begin{equation*}
\hsol_m(t_0):=
\left\{\begin{array}{lcl} 
[e(t_0)]_{\um}^t[\hop]_{\um}^{-1} [\widehat{g}]_{\um}, && \mbox{if $[\hop]_{\um}$ is nonsingular and }\normV{[\hop]^{-1}_{\um}}\leq \sqrt n,\\
0,&&\mbox{otherwise}
\end{array}\right.
\end{equation*}
where $\hsol_m$ is an estimator proposed by \cite{Johannes2009}.
Let $p\geq 3/2$ and $a>0$. Then the estimator $\hsol_{\mstar}(t_0)$ attains within a constant the minimax optimal rate of convergence $\mathcal R^{\rep_{t_0}}$. Applying Proposition~\ref{res:lower:sob} gives
\begin{enumerate}
 \item[(pp)] $\drepnt\sim n^{-(2p-1)/(2p+2a)}$,
 \item[(ep)] $\drepnt\sim (\log n)^{-(2p-1)/(2a)}$,
 \item[(ep)] $\drepnt\sim n^{-1}(\log n)^{(2a+1)/(2p)}$.\hfill$\square$
\end{enumerate}

\end{example}

\begin{example}\label{exp:average}
We want to estimate the average value of the structural function $\sol$ over a certain interval $[0,b]$ with $0<b<1$. The linear functional of interest is given by $\ell_{\rep}(\sol)=\int_0^b \sol(t)dt$ with  representer $\rep:=\1_{[0,b]}$. Its Fourier coefficients are given by
$[\rep]_1=b$, $ [\rep]_{2j}=(\sqrt{2}\pi j)^{-1}\sin(2\pi j b)$, $ [\rep]_{2j+1}=-(\sqrt{2}\pi j)^{-1}\cos(2\pi j b)$ for $j\geq1$ and, hence $[\rep]_{j}^2\sim j^{-2}$.
Again we assume that $p\geq 3/2$ and $a>0$. Then the mean squared error  of the estimator $ \hell_{\mstar}=\int_0^b \hsol_{\mstar}(t)dt$ is bounded up to a constant by the minimax rate of convergence  $\drep$. In the three cases  the order of $\drepn$ is given by
\begin{enumerate}
 \item[(pp)] $\drepn\sim
\left\{\begin{array}{lcl} 
 n^{-(2p+1)/(2p+2a)}, && \mbox{if $a>1/2$},\\
 n^{-1}\log n, && \mbox{if $a=1/2$},\\
 n^{-1}, && \mbox{otherwise,}
\end{array}\right.$
 \item[(ep)]  $\drepn\sim (\log n)^{-(2p+1)/(2a)}$,
 \item[(ep)] $\drepn\sim
\left\{\begin{array}{lcl} 
 n^{-1}(\log n)^{(2a-1)/(2p)}, && \mbox{if $a>1/2$},\\
 n^{-1}\log(\log n), && \mbox{if $a=1/2$},\\
 n^{-1}, && \mbox{otherwise.}
\end{array}\right.$
\end{enumerate}
 As in the  direct regression model where the average value of the regression function can be estimated with  rate $n^{-1}$ we obtain the parametric rate in the case of (pp) and (ep) if $a< 1/2$. 
 \hfill$\square$
\end{example}

\begin{example}
Consider estimation of the weighted average derivative of the structural function $\sol$ with weight function $H$, i.e.,  $\int_0^1\sol'(t)H(t)dt$. This functional is useful not only for estimating scaled coefficients of an index model, but also to quantify the average slope of structural functions. Assume that the weight function $H$ is continuously differentiable and vanishes at the boundary of the support of $Z$, i.e., $H(0)=H(1)=0$. Integration by parts gives
 $\int_0^1\sol'(t)H(t)dt=-\int_0^1\sol(t)h(t)dt=-\ell_{h}(\sol)$  with representer $h$ given by the  derivative of $H$.
The weighted average derivative estimator  $\hell_{\mstar}=-\int_0^1 \hsol_{\mstar}(t)\rep(t)dt$ is minimax optimal.
As an illustration  consider the specific weight function $H(t)=1-(2t-1)^2$ with derivative $h(t)=4(1-2t)$ for $0\leq t\leq 1$. It is easily seen that the Fourier coefficients of the representer $h$ are $[\rep]_1=0$, 
$ [\rep]_{2j}=0$, $ [\rep]_{2j+1}=4\sqrt{2}(\pi j)^{-1}$ for $j\geq1$ and, thus $[\rep]_{2j+1}^2\sim j^{-2}$.  Thus, for the particular choice of the weight function $H$ the estimator $\hell_{\mstar}$ attains up to a constant the optimal rate $\drep$, which  was already specified in Example \ref{exp:average}.
\hfill$\square$
\end{example}

 \section{Adaptive estimation}\label{sec:adaptive}
In this section we derive an adaptive estimation procedure for the value of the linear function $\ell_\rep(\sol)$. This procedure is  based on the estimator $\hell_{\hm}$ given in \eqref{gen:def:est} with dimension parameter $\hm$ selected as a minimizer of the data driven  penalized contrast criterion (\ref{full:def:model}--\ref{full:def:contrast}).
The selection criterion (\ref{full:def:model}--\ref{full:def:contrast}) involves the random upper bound $\hM$ and the random penalty sequence $\hpen$ which we introduce below. We show that the estimator  $\hell_{\hm}$ attains the minimax rate of convergence within a logarithmic term. Moreover, we illustrate the cost due to adaption by considering classical smoothness assumptions.

In an intermediate step we do not consider the estimation of unknown quantities in the penalty function. Let us therefore consider a deterministic upper bound $M_n$ and a deterministic penalty sequence $\pen:=(\pen_m)_{m\geq 1}$, which is nondecreasing. These quantities are constructed such that they can be easily estimated in a second step.
As an adaptive choice $\tm$ of the dimension parameter $m$ we propose the minimizer of  a penalized contrast criterion, that is,
\begin{subequations}
\begin{equation}
  \label{part:def:model}
\tm:=\argmin_{1\leq m\leq M_n}\set{\Psi_{m}+\pen_{m}}
\end{equation}
where  the random sequence of contrast $\Psi:=(\Psi_m)_{m\geq 1}$ is defined by
 \begin{equation}
  \label{part:def:contrast}
\Psi_m:=\max_{m\leq m'\leq M_n}\set{\absV{\hell_{m'}-\hell_{m}}^2 -\pen_{m'}}.
\end{equation}
\end{subequations}
The fundamental idea to establish an appropriate upper bound for the risk of $\hell_{\tm}$ is given by the following reduction scheme. Let us denote $m\wedge m':=\min(m,m')$. 
Due to the definition of $ \Psi$ and $\tm$ we deduce for all $1\leq m\leq M_n$
\begin{multline*}
  \absV{\hell_{\tm}-\ell_\rep(\sol)}^2\leq 3\set{ \absV{\hell_{\tm}-\hell_{\tm\wedge m}}^2+ \absV{\hell_{\tm\wedge m}-\hell_{m}}^2+ \absV{\hell_{m}-\ell_\rep(\sol)}^2}\\
\hfill\leq 3\set{ \Psi_m +\pen_{\tm} +\Psi_{\tm}+\pen_m +  \absV{\hell_{m}-\ell_\rep(\sol)}^2} \\
\hfill\leq 6\set{\Psi_m +\pen_m}+3  \absV{\hell_{m}-\ell_\rep(\sol)}^2
\end{multline*}
where the right hand side does not depend on the adaptive choice $\tm$.
Since the penalty sequence $\pen$ is nondecreasing we obtain
\begin{equation*}
  \Psi_{m}\leq 6\max_{m\leq m'\leq M} \vect{\absV{\hell_{m'}-\ell_\rep(\sol_{m'})}^2 -\frac{1}{6}\pen_{m'}}_+ + 3\max_{m\leq m'\leq M_n} \absV{\ell_\rep(\sol_{m}-\sol_{m'})}^2.
\end{equation*}
Combing the last estimate with the previous reduction scheme yields for all $1\leq m\leq M_n$ 
 \begin{equation} \label{adap:main:ineq}
  \absV{\hell_{\tm}-\ell_\rep(\sol)}^2
\leq7 \pen_m + 78 \bias_m
\hfill+ 42\max_{m\leq m'\leq M} \vect{\absV{\hell_{m'}-\ell_\rep(\sol_{m'})}^2 -\frac{1}{6}\pen_{m'}}_+
\end{equation}
where $\bias_m:=\sup_{m'\geq m} \absV{\ell_\rep(\sol_{m'}-\sol)}^2$.
We will prove below that $\pen_m + \bias_m$ is of the order $\drepnl$. Moreover, we will bound the right hand side term appropriately with the help of Bernstein's inequality.

Let us now introduce the upper bound $M_n$ and sequence of penalty $\pen_m$ used in the penalized contrast criterion (\ref{part:def:model}--\ref{part:def:contrast}).
In the following, assume without loss of generality that $\fou{\rep}_{1}\neq 0$.
\begin{defin}\label{part:def:1}For all $n\geq 1$ let $a_n:=n^{1-1/\log(2+\log n)}(1+\log n)^{-1}$ and $\Mh:=\max\{1\leq m\leq \lfloor n^{1/4}\rfloor:\max\limits_{1\leq j\leq m}\fou{\rep}_{j}^2\leq n\fou{\rep}_{1}^2\}$ then we define
  \begin{equation*}
   M_n:=\min\Big\{2\leq m\leq \Mh:\,m^3\normV{\fou{\op}_{\um}^{-1}}^2 \max\limits_{1\leq j\leq m}[h]_j^2>a_n \Big\}-1
  \end{equation*}
where we set  $M_n:=\Mh$ if the min runs over an empty set. Thus,  $M_n$ takes values between $1$ and $\Mh$.
 Let $\varsigma_m^2=74\big(\Ex[Y^2]+ \max_{1\leq m'\leq m}\|[\op]_\um^{-1}[g]_\um\|^2\big)$, then we define
\begin{equation}\label{part:def:pen}
\pen_m:= 24\,\varsigma_m^2 \,(1+\log n) \,n^{-1} \max_{1\leq m'\leq m}\normV{[\rep]_{\underline m'}^t[\op]_{\underline m'}^{-1}}^2.
\end{equation}
\end{defin}

To apply Bernstein's inequality we need another assumption regarding the error term $U$.
This is captured by the set $\cU^\infty_\sigma$ for some $\sigma>0$, which contains all conditional distributions $P_{U|W}$ such that $\Ex[U|W]=0$, $\Ex[U^2|W]\leq \sigma^2$, and Cramer's condition hold, i.e., 
 \begin{equation*}
  \Ex[|U|^k|W]\leq \sigma^{k}\, k!, \quad k=3,4,\dots.
 \end{equation*}
Moreover, besides Assumption \ref{ass:A1} we need the following Cramer condition which is in particular satisfied if the basis $\{f_l\}_{l\geq 1}$ are uniformly bounded.
\begin{assA}\label{ass:A2}
There exists $\eta\geqslant 1$ such that the distribution of $W$ satisfies  
\begin{itemize}
\item[] $\sup_{j,l\in\N} \Ex| f_j(W)f_l(W)- \Ex [f_j(W)f_l(W)]|^k\leqslant \eta^{k}\, k!$, $k=3,4,\dotsc$.
\end{itemize}
\end{assA}

We now present an upper bound for $\hell_{\tm}$.
 As usual in the context of adaptive estimation of functionals we face a logarithmic loss due to the adaptation.
\begin{theo}\label{partially adaptive unknown operator}
Assume an iid. $n$-sample of $(Y,Z,W)$ from the model (\ref{model:NP}--\ref{model:NP2}) with $\Ex[Y^2]>0$.
Let Assumptions  \ref{ass:reg}--\ref{ass:A2} be satisfied.  Suppose that $(\md)^3\max_{1\leq j\leq\md}[h]_j^2=o( a_n\opw_{\md})$ as $n\to\infty$.
Then we have for all $n\geq1$
 \begin{equation*} \sup_{\op\in\cTdDw}\sup_{P_{U|W}\in\cU_\sigma^\infty} \sup_{\sol \in \Fswsr}\Ex \absV{\hell_{\tm}-\ell_\rep(\sol)}^2\leq \mathcal C \,\drepnl
 \end{equation*}
 for a  constant $\mathcal C>0$ only depending on the classes $\Fswsr$, $\cTdDw$, the constants $\sigma$, $\eta$ and the representer $h$.
\end{theo}

\begin{rem}\label{rem partially adaptive unknown operator}
 In all examples studied below the condition $(\md)^3\max_{1\leq j\leq\md}[h]_j^2=o(a_n\opw_{\md})$ as $n$ tends to infinity is satisfied if the structural function $\sol$ is sufficiently smooth. More precisely, in case of (pp) it suffices to assume $3<2p+2\min(0,s)$. On the other hand, in case of (pe) or (ep) this condition is automatically fulfilled.
\hfill$\square$\end{rem}

In the following definition we introduce empirical versions of the integer $M_n$ and the penalty sequence $\pen$. 
Thereby, we complete the data driven  penalized contrast criterion (\ref{full:def:model}--\ref{full:def:contrast}). 
This allows for a completely data driven selection method.
For this purpose, we construct an estimator for $\varsigma_m^2$ by replacing the unknown quantities by their empirical analogon, that is,
\begin{equation*}
 \widehat\varsigma_m^2:=74\Big(n^{-1}\sum_{i=1}^n Y_i^2+\max_{1\leq m'\leq m}\|[\hop]_\um^{-1}[\widehat g]_\um\|^2\Big).	
\end{equation*}
With the nondecreasing sequence $(\widehat\varsigma_m^2)_{m\geq 1}$ at hand we only need to replace the matrix $[\op]_\um$ by its empirical counterpart (cf. Subsection 3.1).
\begin{defin}\label{full:def:1}
 Let  $a_n$ and $\Mh$ be as in Definition \ref{part:def:1} then for all $n\geq 1$ define
  \begin{equation*}
  \hM:=\min\Big\{2\leq m\leq \Mh:\, m^3\normV{\fou{\hop}_{\um}^{-1}}^2 \max\limits_{1\leq j\leq m}[h]_j^2> a_n\Big\}-1
\end{equation*}
where we set  $\hM:=\Mh$ if the min runs over an empty set. Thus,  $\hM$ takes values between $1$ and $\Mh$.
Then we introduce for all $m\geq 1$ an empirical analogon of $\pen_m$ by
\begin{equation}\label{full:def:pen}
\hpen_m:=204\,\widehat\varsigma_m^2(1+\log n)n^{-1} \max_{1\leq m'\leq m}\normV{[\rep]_{\um'}^t[\rop]_{\um'}^{-1}}^2.
\end{equation}
\end{defin}

Before we establish the next upper bound we introduce
\begin{equation}\label{part:def:bounds:o}
 \Mo:=\min\set{2\leq m\leq \Mh:\,\upsilon_m^{-1}m^3 \max\limits_{1\leq j\leq m}[h]_j^2> 4Da_n}-1
\end{equation}
where  $\Mo:=\Mh$ if the min runs over an empty set. Thus, $\Mo$ takes values between $1$ and $ \Mh$.
As in the partial adaptive case we do not attain up to a constant the minimax rate of convergence $\drep$. A logarithmic term must be paid for adaption as we see in the next assertion.
\begin{theo}\label{fully adaptive unknown operator}Let the assumptions of Theorem \ref{partially adaptive unknown operator} be satisfied.
Additionally suppose that $(\Mo+1)^2\log n=o\big(n\opw_{\Mo+1}\big)$  as $n\to\infty$ and $\sup_{j\geq 1}\Ex|e_j(Z)|^{20}\leq \eta^{20}$.
Then for all $n\geq 1$ we have
\begin{equation*}
\sup_{\op \in \cTdDw}\sup_{P_{U|W}\in\cU_\sigma^\infty} \sup_{\sol \in \Fswsr} \Ex|\hell_{\whm}-\ell_\rep(\sol)|^2   \leq  \mathcal C\,  \drepnl
\end{equation*}
 for a  constant $\mathcal C>0$ only depending on the classes $\Fswsr$, $\cTdDw$, the constants $\sigma$, $\eta$ and the representer $h$.
\end{theo}
\begin{rem}\label{rem fully adaptive unknown operator}
Note that below in all examples illustrating Theorem  \ref{fully adaptive unknown operator} the condition  $(\Mo+1)^2\log n=o(n\opw_{\Mo+1})$ as $n$ tends to infinity is  automatically satisfied.
\hfill$\square$\end{rem}

As in the case of minimax optimal estimation we now present an upper bound uniformly over the class $\Frwrr$ of representer.
For this purpose define $M^\rw_n:=\max\{1\leq m\leq \lfloor n^{1/4}\rfloor:\max_{1\leq j\leq m}(\rw_{j}^{-1})\leq n\}$. In the definition of the bounds $\hM$, $\Mo$, and $\Mu$ (cf. Appendix \ref{sec:adaptive}) we replace $\Mh$ and  $\max_{1\leq j\leq m}[\rep]_j^2$ by $M^\rw_n$ and $\max_{1\leq j\leq m} \omega_j^{-1}$, respectively. Consequently, by employing $\|\rep\|_{1/\sw}^2\leq\rr$ and $\drepn \leq\rr \,\drwn$ for all $\rep \in \cF_\rw^\rr$ the next result follows line by line the proof of Theorem \ref{fully adaptive unknown operator} and hence its proof is omitted.
\begin{coro}\label{fully adaptive unknown operator:coro}
Under the conditions of Theorem \ref{fully adaptive unknown operator} we have for all $n\geq 1$
\begin{equation*} 
 \sup_{\op \in \cTdDw}\sup_{P_{U|W}\in\cU_\sigma^\infty} \sup_{\sol \in \Fswsr,\, \rep\in\Frwrr} \Ex|\hell_{\whm}-\ell_\rep(\sol)|^2   \leq  \mathcal C\, \drwnl
\end{equation*}
where the constant $\mathcal C>0$ depends on the parameter spaces $\Fswsr$, $\Frwrr$, $\cTdDw$, and the constants $\sigma$, $\eta$.
\end{coro}

 \paragraph{Illustration by classical smoothness assumptions.}
Let us illustrate the cost due to adaption by considering classical smoothness assumptions as discussed in Subsection \ref{sec:sob}.
In Theorem~\ref{fully adaptive unknown operator} and Corollary~\ref{fully adaptive unknown operator:coro}, respectively, we have seen that the adaptive estimator $\hell_{\hm}$ attains within a constant the rates $\drepad$ and  $\drwad$.
Let us now present the orders of  these rates by considering the example configurations of Remark \ref{rem:ill:cases}. The proof of the following result is omitted because of the analogy with the proof of Proposition \ref{res:lower:sob}.

\begin{prop}\label{prop:smooth:adapt} Assume an iid. $n$-sample of $(Y,Z,W)$ from the model (\ref{model:NP}--\ref{model:NP2}) with conditional expectation operator $T\in\cTdDw$, error term $U$ such that $P_{U|W}\in\cU_\sigma^\infty$, and $\Ex[Y^2]>0$. Then for the example configurations of Remark \ref{rem:ill:cases} we obtain 
\\[-4ex]
\begin{enumerate}
 \item[(pp)]if in addition $3<2p+2\min(s,0)$ that  $\md\sim \big(n(1+\log n)^{-1}\big)^{1/(2p+2a)}$ and 
\begin{enumerate}
 \item[(i)] 
 $\drepnl\sim
\begin{cases}
(n^{-1}(1+\log n))^{(2p+2s-1)/(2p+2a)}, &\text{if $s-a<1/2$}\\
 n^{-1}(1+\log n)^2, &\text{if $s-a=1/2$}\\
 n^{-1}(1+\log n), &\text{if $s-a>1/2$},
\end{cases}
$ 
\item[(ii)] $\drwnl\sim
\max\big((n^{-1}(1+\log n))^{(p+s)/(p+a)} ,n^{-1}(1+\log n)\big).$
\end{enumerate}
\item[(pe)] $\md\sim \log\big(n(1+\log n)^{-(a+p)/a}\big)^{1/2a}$ 
 and 
\begin{enumerate}
 \item[(i)] 
$\drepnl\sim
(1+\log n)^{-(2p+2s-1)/(2a)}$,
\item[(ii)]
 $\drwnl\sim(1+\log n)^{-(p+s)/a}.
$
\end{enumerate}
\item[(ep)]
 $\md\sim \log\big(n(1+\log n)^{-(a+p)/p}\big)^{1/2p}$ 
 and 
\begin{enumerate}
\item[(i)] 
$ \drepnl\sim
\begin{cases}
 n^{-1}(1+\log n)^{(2a+2p-2s+1)/(2p)}, &\text{if $s-a<1/2$}\\
 n^{-1}(1+\log n)(\log \log n), &\text{if $s-a=1/2$}\\
 n^{-1}(1+\log n), &\text{if $s-a>1/2$},
\end{cases}
$
\item[(ii)] 
$\drwnl\sim\max\big(n^{-1}(\log n)^{(a+p-s)/p} ,n^{-1}(1+\log n)\big).$
\end{enumerate}
\end{enumerate}
\end{prop}

Let us revisit Examples \ref{exp:evaluation} and \ref{exp:average}. In the following, we apply the general theory to adaptive pointwise estimation and adaptive estimation of averages of the structural function $\sol$.

\begin{example}
Consider the point evaluation functional  $\ell_{\rep_{t_0}}(\sol)=\sol(t_0)$, $t_0\in[0,1]$, introduced in Example \ref{exp:evaluation}. 
In this case, the estimator $ \hell_{\hm}$  with dimension parameter $\hm$ selected as a minimizer of criterion  (\ref{full:def:model}--\ref{full:def:contrast}) writes as
\begin{equation*}
\hsol_{\hm}(t_0):=
\left\{\begin{array}{lcl} 
[e(t_0)]_{\underline{\hm}}^t[\hop]_{\underline{\hm}}^{-1} [\widehat{g}]_{\underline{\hm}}, && \mbox{if $[\hop]_{\underline{\hm}}$ is nonsingular and }\normV{[\hop]^{-1}_{\underline{\hm}}}\leq \sqrt n,\\
0,&&\mbox{otherwise}
\end{array}\right.
\end{equation*}
where $\hsol_m$ is an estimator proposed by \cite{Johannes2009}.
 Then $\hsol_{\hm}(t_0)$ attains within a constant the rate of convergence $\mathcal R^{\rep_{t_0}}_{\text{adapt}}$. Applying Proposition~\ref{prop:smooth:adapt} gives
\begin{enumerate}
 \item[(pp)] $\mathcal R^{\rep_{t_0}}_{n(1+\log n)^{-1}}\sim \big(n^{-1}(1+\log n)\big)^{(2p-1)/(2p+2a)}$,
 \item[(ep)] $\mathcal R^{\rep_{t_0}}_{n(1+\log n)^{-1}}\sim (1+\log n)^{-(2p-1)/(2a)}$,
 \item[(ep)] $\mathcal R^{\rep_{t_0}}_{n(1+\log n)^{-1}}\sim n^{-1}(1+\log n)^{(2a+2p+1)/(2p)}$.\hfill$\square$
\end{enumerate}

\end{example}

\begin{example}
Consider the linear functional $\ell_{\rep}(\sol)=\int_0^b \sol(t)dt$ with  representer $\rep:=\1_{[0,b]}$ introduced in Example \ref{exp:average}. 
 The mean squared error  of the estimator $ \hell_{\hm}=\int_0^b \hsol_{\hm}(t)dt$ is bounded up to a constant by  $\drepad$. Applying Proposition~\ref{prop:smooth:adapt} gives
\begin{enumerate}
 \item[(pp)] $\drepnl\sim
\left\{\begin{array}{lcl} 
 (n^{-1}(1+\log n))^{(2p+1)/(2p+2a)}, && \mbox{if $a>1/2$},\\
 n^{-1}(1+\log n)^2, && \mbox{if $a=1/2$},\\
 n^{-1}(1+\log n), && \mbox{otherwise,}
\end{array}\right.$
 \item[(ep)]  $\drepnl\sim (1+\log n)^{-(2p+1)/(2a)}$,
 \item[(ep)] $\drepnl\sim
\left\{\begin{array}{lcl} 
 n^{-1}(1+\log n)^{(2a+2p-1)/(2p)}, && \mbox{if $a>1/2$},\\
 n^{-1}(1+\log n)(\log\log n), && \mbox{if $a=1/2$},\\
 n^{-1}(1+\log n), && \mbox{otherwise.}
\end{array}\right.$
\end{enumerate}

 \hfill$\square$
\end{example}

 \appendix
\section{Appendix}\label{app:proofs}
 \subsection{Proof of the lower bound given in  Section \ref{sec:lower}.}
\begin{proof}[\textcolor{darkred}{\sc Proof of Theorem \ref{res:lower}.}]
 Define  the function $\sol_*:=
\bigg(\frac{\zeta\,\kappa\,\astarn}{\sum_{l=1}^{\mstarn}[\rep]_l^2\Opw_l^{-1}}\bigg)^{1/2}\sum\limits_{j=1}^{\mstarn}[\rep]_j\Opw_j^{-1}\basZ_{j}$ with
$\zeta:=\min(1/(2\Opd),\sr)$. Since $(\sw_j^{-1}\Opw_j)_{j\geq1}$ is nonincreasing and  by using the definition of $\kappa$ given in \eqref{eta:known:operator} it follows
that $\sol_*$ and in particular $\sol_\theta:= \theta \sol_*$ for $\theta\in\{-1,1\}$ belong to $\cF_\sw^\sr$.
 Let $V$ be a Gaussian random variable with mean
zero and variance one ($V \sim \cN(0,1)$) which is independent of
$(Z,W)$. Consider $U_\theta:= [T\sol_{\theta}](W)-\sol_{\theta}(Z) +V $,
then $P_{U_\theta|W}$ belongs to $\cU_\sigma$ for all $ \sigma^4\geq
(\sqrt 3+4\sr \sum_{j\geq 1}\sw_j^{-1}\eta^2)^2$, which can be realized as follows. Obviously, we have $\Ex[
U_\theta|W]=0$. Moreover, we have $\sup_j\Ex[\basZ_j^4(Z)|W]\leq \eta^4$ implies $\Ex[\sol_\theta^4(Z)|W]\leq 
\sr^2 \big(\sum_{j\geq 1}\sw_j^{-1}\big)^2\Ex[\basZ^4_j(Z)|W]\leq
\sr^2\eta^4(\sum_{j\geq 1}\sw_j^{-1})^2$ and thus,
$|[\Op\sol_\theta](W)|^4\leq \Ex
[\sol_\theta^4(Z)|W]\leq \sr^2\eta^4(\sum_{j\geq 1}\sw_j^{-1})^2$. From the
last two bounds we deduce $\Ex [U_\theta^4|W]\leq 16\Ex[\sol_\theta^4(Z)|W]+6\Var(\sol_\theta(Z)|W)+3\leq(\sqrt 3+4\sr\,\eta^2\sum_{j\geq 1}\sw_j^{-1})^2$. Consequently, for each $\theta$
iid. copies $(Y_i,Z_i,W_i)$, $1\leq i\leq n$, of $(Y,Z,W)$ with
$Y:=\sol_{\theta}(Z)+U_\theta$ form an $n$-sample of the model
(\ref{model:NP}--\ref{model:NP2}) and we denote their joint
distribution by $P_{\theta}$ and by $\Ex_\theta$ the expectation with respect to $P_\theta$.  In case of $P_\theta$ the conditional
distribution of $Y$ given $W$ is  Gaussian with mean $
[T\sol_\theta](W)$ and variance $1$.  The log-likelihood of ${P}_{1}$ with respect to ${P}_{-1}$ is given by
\begin{equation*}
\log\Bigl(\frac{d{P}_{1}}{d{P}_{-1}}\Bigr)=\sum_{i=1}^n 2(Y_i - [\Op\sol_*](W_i)) [\Op\sol_*](W_i) + \sum_{i=1}^n  2|[\Op\sol_*](W_i)|^2.
\end{equation*}
Since $\Op\in\Opwd$ the Kullback-Leibler divergence satisfies $KL(P_{1},P_{-1})\leq \Ex_{1}[\log(d{P}_{1}/d{P}_{-1})] =  2n \normV{\Op\sol_*}^2_W\leq2nd \normV{\sol_*}^2_\opw$.
It is well known that the
 Hellinger distance $H(P_{1},P_{-1})$ satisfies
 $H^2(P_{1},P_{-1}) \leqslant KL(P_{1},P_{-1})$  and thus, employing again the definition of $\kappa$  we have
\begin{equation}\label{pr:lower:e3}
H^2(P_{1},P_{-1}) \leqslant  2n \Opd\sum_{j=1}^{\mstarn}[\sol_*]_{j}^2 \Opw_{j} = 2n\Opd\frac{\zeta\,\kappa\,\astarn}{\sum_{l=1}^{\mstarn}[\rep]_l^2\Opw_l^{-1}}\sum\limits_{j=1}^{\mstarn}\frac{[\rep]_j^2}{\Opw_j}
=2\Opd\zeta \frac{\kappa \;\astarn}{n^{-1}}\leq 2\,\Opd\,\zeta \leq 1.
\end{equation} 
Consider the  Hellinger affinity $\rho(P_{1},P_{-1})= \int \sqrt{dP_{1}dP_{-1}}$ then  for any estimator $\breve{\ell}$ it holds 
\begin{align}\nonumber
\rho(P_{1},P_{-1})&\leqslant 
\int \frac{|\breve{\ell}-\ell_{h}(\sol_1)|}{2|\ell_{h}(\sol_*)|} 
\sqrt{dP_1dP_{-1}} 
+
\int \frac{|\breve{\ell}-\ell_{h}(\sol_{-1})|}{2|\ell_{h}(\sol_*)|} 
\sqrt{dP_{1}dP_{-1}} 
\\\label{pr:lower:e4}
&\leq
\Bigl( \int  \frac{|\breve{\ell}-\ell_{h}(\sol_{1})|^2}
{4|\ell_{h}(\sol_*)|^2}dP_{1}\Bigr)^{1/2}
+
\Bigl( 
\int  \frac{|\breve{\ell}-\ell_{h}(\sol_{-1})|^2}
{4|\ell_{h}(\sol_*)|^2}dP_{-1}\Bigr)^{1/2}.
\end{align}
Due to the identity $\rho(P_{1},P_{-1})=1-\frac{1}{2}H^2(P_{1},P_{-1})$   combining  \eqref{pr:lower:e3} with 
 \eqref{pr:lower:e4} yields
\begin{equation}\label{pr:lower:e5}
\Ex_{1}|
\breve{\ell}
-
\ell_{h}(\sol_{1})|^2
+ 
\Ex_{-1}
|\breve{\ell}
-
\ell_{h}(\sol_{-1})|^2
\geqslant
\frac{1}{2}
|\ell_{h}(\sol_*)|^2.
 \end{equation}
Obviously,  $|\ell_{\rep}(\sol_*)|^2
=\zeta \kappa \astarn \sum\limits_{j=1}^{\mstarn}[\rep]_j^2\Opw_j^{-1}$. From \eqref{pr:lower:e5} together with the last identity    we conclude  for any possible estimator $\breve\ell$
\begin{align}\nonumber
 \sup_{\op\in\OpwdD}\sup_{P_{U|W}\in \cU_\sigma} \sup_{\sol\in \cF_\sw^\sr} &\Ex|\breve{\ell}-\ell_{\rep}(\sol)|^2 \geqslant \sup_{\theta\in \{-1,1\}} \Ex_{\theta}|\breve{\ell} -\ell_{\rep}(\sol_*^{(\theta)})|^2\\\nonumber
&\geqslant \frac{1}{2}
\Bigl\{\Ex_{1}|
\breve{\ell}-\ell_{\rep}(\sol_{1})|^2+ \Ex_{-1}|\breve{\ell}-\ell_{\rep}(\sol_{-1})|^2
\Bigr\}\\\label{eq:2}
&\geqslant \frac{ \kappa}{4}\,\min \bigg( \frac{1}{2 \Opd},  \sr\bigg)\,\astarn\sum\limits_{j=1}^{\mstarn}[\rep]_j^2\Opw_j^{-1}.
\end{align}
Consider now   $\widetilde\sol_*:=
\bigg(\frac{\zeta\,\kappa}{\sum_{l>\mstarn}[\rep]_l^2\sw_l^{-1}}\bigg)^{1/2}\sum\limits_{j>\mstarn}[\rep]_j\sw_j^{-1}\basZ_{j}$, which belongs to $\cF_\sw^\sr$ since $\kappa\leq 1$ and $\zeta\leq \sr$.
Moreover,  since $(\sw_j^{-1}\Opw_j)_{j\geq1}$ is nonincreasing and  by using the definition of $\kappa$ given in \eqref{eta:known:operator} we have 
\begin{equation*}\label{eq:4}
2n \Opd\sum_{j>\mstarn}[\widetilde\sol_*]_{j}^2 \Opw_{j}=  2n\Opd\frac{\zeta\,\kappa}{\sum_{l>\mstarn}[\rep]_l^2\sw_l^{-1}}\sum\limits_{j>\mstarn}\frac{[\rep]_j^2\Opw_j}{\sw_j^2}
\leq 2\Opd\zeta \frac{\kappa}{\sw_{\mstarn}\Opw_{\mstarn}^{-1}}\leq 2\,\Opd\,\zeta \leq 1.
\end{equation*} 
Thereby, following line by line the proof of \eqref{eq:2} we obtain for any possible estimator $\breve\ell$
\begin{align}\nonumber
\sup_{\op\in\OpwdD}\sup_{P_{U|W}\in \cU_\sigma} \sup_{\sol\in \cF_\sw^\sr} &\Ex|\breve{\ell}-\ell_{\rep}(\sol)|^2 \geqslant \frac{1}{4}|\ell_{h}(\widetilde\sol_*)|^2
= \frac{ \kappa}{4}\,\min \bigg( \frac{1}{2 \Opd},  \sr\bigg)\,\sum\limits_{j>\mstarn}[\rep]_j^2\sw_j^{-1}.
\end{align}
Combining, the last estimate and (\ref{eq:2}) implies the result of the theorem,  which completes the proof. \end{proof}

\subsection{Proofs of Section \ref{sec:gen}.}\label{app:proofs:gen}
 We begin by defining and recalling notations to be used in the proofs of this section. For $m\geq 1$ recall $\sol_m=\sum_{j=1}^m[\sol_m]_{j}\basZ_j$ with
 $[\sol_m]_{\um}=[\Op]_{\um}^{-1}[g]_{\um}$  keeping in mind that $[T]_{\um}$ is  nonsingular. Then the identities  $[T(\sol-\sol_m)]_{\um}=0$ and $[\sol_m-E_m \sol]_{\um} = [T]_{\um}^{-1}[TE_m^\perp \sol]_{\um}$ hold true.  
We denote $\Xi_m:= [\widehat T]_{\um}- [T]_{\um}$ and $V_m:=[\widehat{g}]_{\um}- [\widehat T]_{\um} [\sol_m]_{\um}=n^{-1}\sum_{i=1}^n(U_i+\sol(Z_i)-\sol_m(Z_i))[f(W_i)]_{\um}$, where obviously $\Ex V_m= 0$.
Moreover, let us introduce the events 
 \begin{multline*}
\Omega_m:=\{ \normV{[\widehat{T}]^{-1}_{\um}}\leq \sqrt{n}\},\quad 
\mho_m:= \{\sqrt m \normV{\Xi_m}\normV{[T]_{\um}^{-1}}\leq 1/2\}\\
\Omega_m^c:=\{ \normV{[\widehat{T}]^{-1}_{\um}}> \sqrt{n}\}\quad\mbox{ and }\quad  \mho_m^c=\{\sqrt m\normV{\Xi_m}\normV{[T]_{\um}^{-1}}> 1/2\}.
\end{multline*}
 Observe that if $\sqrt m\normV{\Xi_m}\normV{[T]_{\um}^{-1}}\leqslant 1/2$  then the identity $[\hop]_{\um}= [\Op]_{\um}\{I+[\Op]^{-1}_{\um}\Xi_m\}$ implies by the usual Neumann series argument that $\normV{[\hop]^{-1}_{\um}} \leqslant 2\normV{[\Op]^{-1}_{\um}}$. Thereby, if $\sqrt{n} \geqslant 2\normV{[\Op]^{-1}_{\um}}$ we have $\mho_m \subset\Omega_m$. These results will be used below without further reference.
We shall prove at the end of this section four technical Lemmata (\ref{app:gen:upper:l2} -- \ref{pr:minimax:l2}) which are used in the following proofs. 
Furthermore, we will denote by $C$ universal numerical constants and by $C(\cdot)$ constants depending only on the arguments. In both cases, the values of the constants may change from line to line.
\paragraph{Proof of the consistency.}
\begin{proof}[\textcolor{darkred}{\sc Proof of Proposition \ref{res:gen:prop:cons}.}] 
Consider for all $m\geq 1$ the decomposition
  \begin{multline}\label{pr:cons:e1}
\Ex |\hell_m  -\ell_\rep(\sol)|^2= \Ex |\hell_m  -\ell_\rep(\sol)|^2\1_{\Omega_m} + |\ell_\rep(\sol)|^2P(\Omega_m^c)\\\hfill\leq 2 \Ex|\hell_m  -\ell_\rep(\galsol)|^2\1_{\Omega_m}
+ 2 |\ell_\rep(\galsol-\sol)|^2 +|\ell_\rep(\sol)|^2P(\Omega_m^c)
  \end{multline}
where we bound each term separately. Let 
$\overline\mho_m:= \{ \normV{\Xi_m}\normV{[T]_{\um}^{-1}}\leq 1/2\}$
and let $\overline\mho_m^c$ denote its complement.
By employing   $\normV{[\hop]_{\um}^{-1}}\1_{\overline\mho_m}\leq 2\normV{[\op]_{\um}^{-1}}$ and
$\normV{[\hop]_{\um}^{-1}}^2\1_{\Omega_m}\leq n$ it follows that
\begin{multline*}
  |\hell_m-\ell_\rep(\galsol)|^2\1_{\Omega_m}\leq 2\abs{[\rep]_{\um}^t[\op]_{\um}^{-1}V_m}^2+
  2\big|[\rep]_{\um}^t[\op]_{\um}^{-1}\Xi_m[\hop]_{\um}^{-1}V_m\big|^2\1_{\Omega_m}(\1_{\overline\mho_m}+\1_{\overline\mho_m^c})\\
\leq 2|[\rep]_{\um}^t[\op]_{\um}^{-1}V_m|^2 +  2\normV{[\rep]_{\um}^t[\op]_{\um}^{-1}}^2 \set{4\normV{[\op]_{\um}^{-1}}^2 \normV{\Xi_m}^2\normV{V_m}^2 + n \normV{\Xi_m}^2\normV{V_m}^2\1_{\overline\mho_m^c}}.
\end{multline*}
Thus, from estimate \eqref{app:gen:upper:l2:e1:1}, \eqref{app:gen:upper:l2:e2:1}, and \eqref{app:gen:upper:l2:e3} in Lemma \ref{app:gen:upper:l2} we infer
\begin{multline}\label{pr:cons:e3}
  \Ex|\hell_m-\ell_\rep(\galsol)|^2\1_{\Omega_m}\leq C(\sw) n^{-1} \normV{[\rep]_{\um}^t[\op]_{\um}^{-1}}^2 \eta^4\big(\sigma^2 +\normV{\sol-\galsol}_\sw^2\big)\\
\times \Big\{1+  \frac{m^3}{n}\normV{[\op]_{\um}^{-1}}^2+ m^3P^{1/4}(\overline\mho_m^c)\Big\}.
\end{multline}
Let $m=m_n$ satisfying $m_n^{-1}=o(1)$, $m_n=o(n)$, and condition \eqref{eq:5}. 
We have $ \sqrt{n}\geq 2  \normV{[\Op]_{\umn}^{-1}}$ and thus,
$\Omega_{m_n}^c\subset \overline\mho_{m_n}^c$ for $n$ sufficiently large. From Lemma \ref{pr:minimax:l1} it follows that 
$m_n^{12}P(\overline\mho_{m_n}^c)\leq  2\exp\big\{- m_n\,(32\eta^2n^{-1}{m_n}^3\normV{[\Op]_{\umn}^{-1}}^2)^{-1} +14\log m_n \big\}=O(1)$ as $n\to\infty$ since $m_n(4n^{-1}m_n^3\normV{[\Op]_{\umn}^{-1}}^2)^{-1}\leq 4\eta^2 n$ for $n$ sufficiently large. Thus, in particular
$P(\Omega_{m_n}^c)=o(1)$. Consequently, as $n\to\infty$ we obtain $\Ex|{\hell_{m_n}}-\ell_\rep(\sol_{m_n})|^2\1_{\Omega_{m_n}}=o(1)$ since $\normV{[\rep]_{\umn}^t [T]_{\umn}^{-1}}^2=o(n)$. Moreover, as $n\to\infty$ it holds $|\ell_\rep(\sol_{m_n})-{\ell_\rep(\sol)}|^2\leq \|\rep\|_{1/\sw}\normV{\sol-\sol_{m_n}}_\sw=o(1)$  due to condition \eqref{eq:cons:solm}, and $|\ell_\rep(\sol)|^2P(\Omega_{m_n}^c)\leq \|\rep\|_{1/\sw}\|\sol\|_{\sw}P(\Omega_{m_n}^c)=o(1)$. This together with decomposition \eqref{pr:cons:e1} proves the result.
\end{proof}

\begin{proof}[\textcolor{darkred}{\sc Proof of Corollary \ref{res:gen:coro:cons}.}] The assertion follows directly from Proposition \ref{res:gen:prop:cons}, it only remains to
  check  conditions \eqref{eq:cons:solm} and \eqref{eq:5}.
 We make use of decomposition $\normV{\sol-\sol_m}_\sw\leq \normV{E_m^\perp
  \sol}_\sw +\normV{E_m \sol -\sol_m}_\sw $. 
As in the proof of  Lemma \ref{app:gen:upper:l3} we conclude   $\|E_m\sol-\sol_m\|_\sw^2\leqslant\normV{E_m^\perp \sol   }_\sw \sup_m \sup_{\normV{\phi}_\sw=1}\normV{T^{-1}_{m} F_m T E_m^\perp \phi}_\sw\leq Dd\normV{E_m^\perp \sol}_\sw$.
By using Lebesgue's dominated convergence theorem we observe $\normV{E_m^\perp  \sol}_\sw=o(1)$ as $m\to\infty$ and hence \eqref{eq:cons:solm} holds.
Condition $\Op\in\OpwdD$ implies $\normV{[\rep]_{\um}^t [T]_{\um}^{-1}}^2\leq \OpD \sum_{j=1}^m{[\rep]_{j}^2}{\Opw_j}^{-1}$ and  $\normV{[\Op]^{-1}_{\um}}^2\leqslant \OpD\Opw_m^{-1}$ for all $m\geq 1$ since  $\Opw$ is nonincreasing.
 Thereby, condition \eqref{eq:6} implies  condition \eqref{eq:5}, which completes the proof.
\end{proof}

 \paragraph{Proof of the upper bound.}
\begin{proof}[\textcolor{darkred}{\sc Proof of Theorem \ref{res:upper}.}] The proof is based on  inequality \eqref{pr:cons:e1}.
Applying estimate  \eqref{app:gen:upper:l3:e2} in Lemma \ref{app:gen:upper:l3}  gives  
$\absV{\ell_\rep(\galsol-\sol)}^2\leq 2\sr\,\{\sum_{j>m}[\rep]_j^2\sw_j^{-1}+
\opD\opd\,\opw_m\sw_m^{-1}\sum_{j=1}^m[\rep]_j^2\opw_j^{-1}\}$ for all $\sol\in\Fswsr$   and $\rep\in\cF_{1/\sw}$. Since $|\ell_\rep(\sol)|^2\leq \normV{\sol}_\sw^2\normV{\rep}_{1/\sw}^2$ and $\normV{\sol}_\sw^2\leq \sr$  we conclude
 \begin{multline}\label{pr:theo:upper:e2}
\Ex |\hell_m  -\ell_\rep(\sol)|^2\leq 2 \Ex|\hell_m  -\ell_\rep(\galsol)|^2\1_{\Omega_m}\\
+4\sr\,\Big\{\sum_{j>m}[\rep]_j^2\sw_j^{-1}+
\opd\opD\frac{\opw_m}{\sw_m}\sum_{j=1}^m[\rep]_j^2\opw_j^{-1}\Big\}+
\sr\normV{\rep}_{1/\sw}^2P(\Omega_m^c).
  \end{multline}

By employing   $\normV{\Xi_m[\rop]_{\um}^{-1}}^2\1_{\mho_m}\leq m^{-1}$ and
$\normV{[\hop]_{\um}^{-1}}^2\1_{\Omega_m}\leq n$ it follows that
\begin{multline*}
  |\hell_m-\ell_\rep(\galsol)|^2\1_{\Omega_m}
\leq 2|[\rep]_{\um}^t[\op]_{\um}^{-1}V_m|^2 + 2m^{-1}  \normV{[\rep]_{\um}^t[\op]_{\um}^{-1}}^2 \normV{V_m}^2\\
  \hfill+ 2n \normV{[\rep]_{\um}^t[\op]_{\um}^{-1}}^2 \normV{\Xi_m}^2\normV{V_m}^2\1_{\mho_m^c}.
\end{multline*}

Due to $\op\in\opwdD$  and $\sol\in\Fswsr$ we have $\normV{[\rep]_{\um}^t[\op]_{\um}^{-1}}^2\leq
\opD\sum_{j=1}^m\fou{\rep}_j^2/\opw_j$ and $\normV{\sol-\galsol}_\sw^2\leq 2\,\sr\,(1+ \,\opD \, \opd)$ (cf. \eqref{app:gen:upper:l3:e1} in Lemma \ref{app:gen:upper:l3}), respectively. Thereby, similarly to the proof of Proposition \ref{res:gen:prop:cons} we get 
\begin{equation*}
  \Ex|\hell_m-\ell_\rep(\galsol)|^2\1_{\Omega_m}\leq C(\sw) \opD (\sigma^2+\eta^2\opd\opD\sr)  n^{-1}\sum_{j=1}^m\fou{\rep}_j^2\opw_j^{-1}\set{1+ m^3 P(\mho_m^c)^{1/4}}.
\end{equation*}
Combining the last estimate with \eqref{pr:theo:upper:e2} yields
 \begin{multline}\label{pr:theo:upper:e3}
\Ex |\hell_m  -\ell_\rep(\sol)|^2\leq  C(\sw) \opD (\sigma^2+\eta^2\opd\opD\sr) \max\Big\{\sum_{j>m}[\rep]_j^2\sw_j^{-1}, \max\Big(\frac{\opw_m}{\sw_m}, n^{-1}\Big)\sum_{j=1}^m[\rep]_j^2\opw_j^{-1}\Big\}\\
\times\set{1 +  m^3 P(\mho_m^c)^{1/4}} +\sr\normV{\rep}_{1/\sw}^2P(\Omega_m^c).
  \end{multline}
Consider now the optimal choice $m=\mstarn$ defined in \eqref{m:known:operator}, then we have 
 \begin{multline*}
\Ex |\hell_{\mstarn}  -\ell_\rep(\sol)|^2\leq C(\sw) \opD \set{\sigma^2+\sr\big(\eta^2\opd\opD+\normV{\rep}_{1/\sw}^2\big)}  \drepn\\
\times \set{1+ (\mstarn)^3 P(\mho_{\mstarn}^c)^{1/4}  + (\drepn)^{-1}P(\Omega_{\mstarn}^c)}
  \end{multline*}
and hence, the assertion follows by making use of Lemma \ref{pr:minimax:l2}.
\end{proof}

 \paragraph{Technical assertions.}\hfill\\[1ex]
The following paragraph gathers technical results used in the proofs of Section  \ref{sec:gen}. Below we consider the set $\mmS^m:=\{s\in\R^m:\|s\|=1\}$.
\begin{lem}\label{app:gen:upper:l2} Suppose that $P_{U|W}\in \cU_\sigma$ and that the joint distribution of  $(Z,W)$ satisfies Assumption \ref{ass:A1}.  If in addition $\sol\in \cF_\bw^\br$ with $\sw$ satisfying Assumption \ref{ass:reg}, then  for all $m\geq 1$ we have
\begin{gather}\label{app:gen:upper:l2:e1:1}
 \sup_{s\in \mmS^m} \Ex|s^t \,V_m|^{2} \leqslant 2n^{-1}\big( \sigma^2+C(\sw)\,\eta^2\normV{ \sol -   \sol_m}_\bw^2\big ),\\
\label{app:gen:upper:l2:e2:1}
\Ex\normV{V_m}^{4}\leqslant C(\sw)\,\big(n^{-1}m\,\eta^2(\sigma^{2}+\normV{\sol-\sol_m}_\bw^2)\big)^2,
\\\label{app:gen:upper:l2:e3}
 \Ex\normV{\Xi_m}^8\leqslant C \, \big(n^{-1}m^2\,  \eta^{2}\bigr)^4.
\end{gather}
\end{lem}
 \begin{proof}[\textcolor{darkred}{\sc Proof.}] 
Proof of \eqref{app:gen:upper:l2:e1:1}.
Since $(\{ U_i+\sol(Z_i) -   \sol_m(Z_i)\} \sum_{j=1}^m s_j f_{j}(W_i))$,  $1\leqslant i\leqslant n,$ are iid. with mean zero we have $\Ex|s^t \,V_m |^{2}=n^{-1}\Ex |\{U+ \sol(Z) -   \sol_m(Z)\} \sum_{j=1}^m s_j f_{j}(W)|^2$. Then  \eqref{app:gen:upper:l2:e1:1} follows from $\Ex [U^2|W]\leqslant (\Ex[U^4|W])^{1/2}\leqslant \sigma^2$ and  from Assumption \ref{ass:A1} (i), i.e., $\sup_{j\in\N}\Ex [e_j^2(Z)|W] \leqslant \eta^2$. Indeed, applying condition $|j|^3\bw_j^{-1}=o(1)$ (cf. Assumption \ref{ass:reg}) gives $\sum_{j\geq 1}\sw_j^{-1}\leq C(\sw)$ and thus,
\begin{multline*}\Ex |\{ \sol(Z) -   \sol_m(Z)\} \sum_{j=1}^m s_j f_{j}(W)|^2\leqslant \normV{\sol-\sol_m}_\bw^2 \sum_{l=1}^\infty\bw_l^{-1} \Ex |e_l(Z)\sum_{j=1}^m s_j f_{j}(W)|^2\\\leq C(\sw)\, \eta^2 \normV{\sol-\sol_m}_\bw^2 \,\sum_{j=1}^m s_j^2 = C(\sw)\,\eta^2\normV{\sol-\sol_m}_\bw^2.
\end{multline*}

Proof of  \eqref{app:gen:upper:l2:e2:1}. Observe  that  for each $1\leqslant j \leqslant m$,  $(\{U_i+\sol(Z_i) -   \sol_m(Z_i)\}f_j(W_i))$,  $1\leqslant i\leqslant n,$ are iid. with mean zero. It follows from  Theorem 2.10 in \cite{Petrov1995} that $\Ex\normV{V_m}^{4}\leq C n^{-2}m^2 \sup_{j\in\N}\Ex | \{U+\sol(Z) -   \sol_m(Z)\} f_j(W)|^4$. Thereby, \eqref{app:gen:upper:l2:e2:1} follows from $\Ex [U^4|W]\leqslant \sigma^4$ and $\sup_{j\in\N}  \Ex[f_j^4(W)]\leq \eta^4$ together with $\Ex |\{ \sol(Z) -  \sol_m(Z)\}f_{j}(W)|^4\leqslant  C(\sw)\,\eta^4  \normV{\sol-\sol_m}_\bw^4$, which can be realized as follows. Since $[T(\sol-\sol_m)]_j=0$ we have
$\{\sol(Z) - \sol_m(Z)\}f_j(W)= \sum_{l\geq 1} [\sol-\sol_m]_l \{e_l(Z)f_j(W)-[T]_{j,l}\}$. Furthermore, Assumption \ref{ass:A1} (ii), i.e., $\sup_{j,l\in\N}  \Ex |e_l(Z)f_j(W)-[T]_{j,l}|^4\leq 4!\eta^4$,  implies%
\begin{multline*}\Ex |\{ \sol(Z) - \sol_m(Z)\}f_{j}(W)|^4\leqslant 
\normV{\sol-\sol_m}_\bw^4 \Ex\Bigl|\sum_{l\geq 1}\bw_l^{-1} |e_l(Z)f_{j}(W)-[T]_{j,l}|^2\Bigr|^2\\
\leq C(\sw)\,\eta^4 \normV{\sol-\sol_m}_\bw^4.\end{multline*}

Proof of \eqref{app:gen:upper:l2:e3}.  The random variables $(e_{l}(Z_i)f_j(W_i)-[T]_{j,l})$,  $1\leqslant i\leqslant n,$ are iid. with mean zero for each $1\leqslant j,l \leqslant m$. Hence, Theorem 2.10 in \cite{Petrov1995} implies $\Ex\normV{ \Xi_m}^{8}\leq C n^{-4}m^{8} \sup_{j,l\in\N}  \Ex |e_l(Z)f_j(W)-[T]_{j,l}|^8$ and thus, the assertion follows from  Assumption \ref{ass:A1} (ii), which completes the proof.\end{proof}

\begin{lem}\label{app:gen:upper:l3}
 If $\Op\in\OpwdD$ and  $\sol\in \cF_\sw^\sr$, then for all  $m\geq1$  we have
\begin{align} \label{app:gen:upper:l3:e1:1}
\normV{E_m \sol-\sol_m}^2_{\sw}&\leqslant\OpD\,\Opd\,\sr,\\\label{app:gen:upper:l3:e1}
\normV{\sol-\sol_m}_{\sw}^2&\leqslant 2\,(1+ \,\OpD \, \Opd)\, \sr,\\
  \label{app:gen:upper:l3:e2}
|\skalarV{\rep,\sol-\sol_m}_Z|^2&\leqslant 2\,\sr\, \sum_{j>m} \frac{[\rep]_j^2}{\sw_j}  + 2\,\OpD\,\Opd\, \sr\,\frac{\Opw_m}{\sw_m} \sum_{j=1}^m\frac{[\rep]^2_j}{\Opw_j}.
\end{align}
\end{lem}
\begin{proof}[\textcolor{darkred}{\sc Proof.}]  Consider \eqref{app:gen:upper:l3:e1:1}. 
Since  $\Op\in\OpwdD$ the identity $[E_m \sol-\sol_m]_{\um} = -[\Op]_{\um}^{-1}[\Op E_m^\perp \sol]_{\um}$ implies
$\normV{E_m\sol-\sol_m}^2_\Opw \leqslant\OpD \normV{ \Op E_m^\perp \sol}^2_W\leq \OpD\Opd\normV{E_m^\perp \sol}_\Opw^2$. Consequently, 
\begin{equation}\label{app:gen:upper:l3:e1:2}
\normV{E_m\sol-\sol_m}^2_\Opw \leqslant\OpD\, \Opd\,   \sw_m^{-1}\Opw_m \normV{ \sol}^2_\sw
\end{equation}
because $(\sw_j^{-1}\Opw_j)_{j\geq1}$ is nonincreasing and thus, $ \normV{E_m\sol-\sol_m}^2_{\sw}
\leqslant \sw_m\Opw_m^{-1}\,\normV{E_m\sol-\sol_m}^2_\Opw$. By combination of the last estimate and
\eqref{app:gen:upper:l3:e1:2} we obtain the assertion  \eqref{app:gen:upper:l3:e1:1}.  By employing the decomposition $\normV{\sol-\sol_m}^2_{\sw}\leqslant 2  \normV{\sol-E_m \sol}^2_{\sw} +2\normV{E_m \sol-\sol_m}^2_{\sw}$
the bound \eqref{app:gen:upper:l3:e1} follows from \eqref{app:gen:upper:l3:e1:1} and $ \normV{\sol-E_m\sol}^2_{\sw} \leqslant \normV{\sol}^2_\sw$. It remains to show \eqref{app:gen:upper:l3:e2}. Applying the Cauchy-Schwarz inequality gives $|\skalarV{\rep,\sol-E_m\sol}_Z|^2\leqslant \normV{\sol}_\sw^2  \sum_{j>m} [\rep]_j^2\sw_j^{-1}$
and $|\skalarV{\rep,E_m\sol-\sol_m}_Z|^2\leq\OpD\,\Opd\, \normV{\sol}_\sw^2\,\Opw_m\sw_m^{-1} \sum_{j=1}^m[\rep]^2_j\Opw_j^{-1}$ by \eqref{app:gen:upper:l3:e1:2}.
Thereby \eqref{app:gen:upper:l3:e2} follows from the inequality $|\skalarV{\rep,\sol-\sol_m}_Z|^2\leqslant 2|\skalarV{\rep,\sol-E_m\sol}_Z|^2 + 2|\skalarV{\rep,E_m\sol-\sol_m}_Z|^2 $, which completes the proof.\end{proof}

\begin{lem}\label{pr:minimax:l1} Suppose that the joint distribution of  $(Z,W)$ satisfies Assumption \ref{ass:A1}. Then for all $n\geq 1$ and $m\geq 1$ we have
\begin{equation}\label{pr:minimax:l1:e1}
P\big(m^{-2}n\normV{\Xi_m}^2\geq t \big)\leqslant 2\exp\big(-\frac{t}{8\eta^2} +2\log m \big)\quad \text{for all } 0<t\leq 4\,\eta^2 n. 
\end{equation}
\end{lem}
\begin{proof}
 Our proof starts with the observation that for all $j,l\in\N$ the condition (ii) in Assumption \ref{ass:A1} implies for all $t>0$
\begin{gather*}
P\big(\big|\sum_{i=1}^n\{e_j(Z_i)f_l(W_i)- \Ex [e_j(Z)f_l(W)]\}\big| \geqslant t \big)\leqslant 2 \exp\bigg( \frac{-t^2}{4 n \eta^2 + 2\eta t }\bigg),
\end{gather*} which
is just Bernstein's inequality (cf. \cite{Bosq1998}). This implies for all $0<t\leq 2\eta n$ 
\begin{gather}\label{app:gen:upper:l4:e1:1}
\sup_{j,l\in\N}P\big(\big|\sum_{i=1}^n\{e_j(Z_i)f_l(W_i)- \Ex [e_j(Z)f_l(W)]\}\big| \geqslant t \big)\leqslant 2 \exp\big(-\frac{t^2}{8\eta^2 n}\big).
\end{gather}
It is well-known that $m^{-1}\normV{[A]_{\um}}\leqslant \max_{1\leqslant j,l\leqslant m }|[A]_{j,l}|$ for any $m\times m$ matrix $[A]_{\um}$.
Combining the last estimate and \eqref{app:gen:upper:l4:e1:1} we obtain for all $0<t\leq 2\eta\, n^{1/2}$
\begin{multline*}
P\big(m^{-1}n^{1/2}\normV{\Xi_m}\geqslant t\big)\leqslant \sum_{j,l=1}^m P\Big(\big|\sum_{i=1}^n\big(e_j(Z_i)f_l(W_i)- \Ex [e_j(Z)f_l(W)]\big)\big| \geqslant n^{1/2} t \Big)\\ \leqslant 2 \exp\big(-\frac{t^2}{8\eta^2} + 2 \log m\big).
\end{multline*}
\end{proof}

\begin{lem}\label{pr:minimax:l2} Under the conditions of Theorem \ref{res:upper} we have for all $n\geq 1$ 
\begin{gather}
(\mstarn)^{12} P(\mho_{\mstarn}^c) \leq C(\sw,\opw,\eta,D) \label{pr:minimax:l2:eq1}\\
(\drepn)^{-1}P(\Omega_{\mstarn}^c)\leq C(\sw,\opw,\eta,\rep,D)\label{pr:minimax:l2:eq2}.
\end{gather}
\end{lem}
 \begin{proof}[\textcolor{darkred}{\sc Proof.}]
Proof of \eqref{pr:minimax:l2:eq1}.  Since $\normV{[\op]^{-1}_{\um}}^2\leq \opD\opw_m^{-1}$ due to $\op\in\opwdD$ it follows from Lemma   \ref{pr:minimax:l1} for all $m,n\geq 1$ that 
\begin{equation*}
 P(\mho_{m}^c)\leq
P\Big(m^{-2}n\normV{\Xi_m}^2>\frac{n\opw_m}{4\opD m^{3}}\Big)\leq 2\exp\Big(-  \frac{n\opw_m}{32 \opD  \eta^2 m^3} +2\log m\Big)
\end{equation*}
since $(4\opD m^{3}\opw_m^{-1})^{-1}\leq 1 \leq 4\eta^2$ for all  $m\geq 1$. Due to condition \eqref{minimax upper c1}  there exists $n_0\geq 1$
  such that     $n\opw_{\mstarn}\geq 448\opD\eta^2(\mstarn)^3 \log\mstarn$
  for all $n\geq n_0$. Consequently, $  (\mstarn)^{12} P(\mho_{\mstarn}^c) \leq 2$ for all
$n\geq n_0$, while trivially $ (\mstarn)^{12} P(\mho_{\mstarn}^c) \leq (m^*_{n_0})^{12}$ for all $n\leq n_0$, which gives \eqref{pr:minimax:l2:eq1} since $n_0$ and $m^*_{n_0}$ depend on $\sw$, $\opw$, $\eta$ and $D$ only.

Consider \eqref{pr:minimax:l2:eq2}. Let $n_0\in\N$ such that $\max\{|\log \drepn |, (\log \mstarn)\} (\mstarn)^3 \leq n\opw_{\mstarn}(96\opD\eta^2)^{-1}$ for all $n\geq n_0$. Observe that $\mho_{m} \subset\Omega_m$  if $n \geqslant 4\opD \opw_m^{-1}$.
Since $(\mstarn)^{-3}n \opw_{\mstarn}\geq 96\opD\eta^2$ for all $n\geq n_0$  it follows $n\opw_{\mstarn}\geq 4\opD$ for all
$n\geq n_0$ and hence $( \drepn)^{-1}P(\Omega_{\mstarn}^c)\leq ( \drepn)^{-1} P(\mho_{\mstarn}^c) \leq2$ for all $n\geq n_0$ as in the proof of \eqref{pr:minimax:l2:eq1}.
Combining the last estimate and the elementary inequality $( \drepn)^{-1}P(\Omega_{\mstarn}^c)\leq (\drepno)^{-1}$ for all  $n\leq n_0$  shows  \eqref{pr:minimax:l2:eq2} since $n_0$ depends on $\sw$, $\opw$, $\eta$, $\rep$ and $D$ only, which completes the proof.\end{proof}

\subsection{Proofs of Section \ref{sec:sob}}\label{app:proofs:sob}
 
\begin{proof}[\textcolor{darkred}{\sc Proof of Proposition \ref{res:lower:sob}.}]

Proof of (pp). From the definition of $\mstarn$ in \eqref{m:known:operator} it follows  $\mstarn\sim n^{1/(2p+2a)}$.
 Consider case (i).
The condition $s-a<1/2$ implies $n^{-1}\sum_{j=1}^{\mstarn} |j|^{2a-2s}\sim n^{-1}(\mstarn)^{2a-2s+1}\sim n^{-(2p+2s-1)/(2p+2a)}$ and moreover,  $\sum_{j>\mstarn}|j|^{-2p-2s}\sim n^{-(2p+2s-1)/(2p+2a)}$ since $p+s>1/2$. If  $s-a=1/2$ then
$n^{-1}\sum_{j=1}^{\mstarn} |j|^{2a-2s}\sim n^{-1}\log( n^{1/(2p+2a)})$ and
$\sum_{j>\mstarn}|j|^{-2p-2s}\sim n^{-1}$. In the case of $s-a>1/2$ it follows that  $\sum_{j=1}^{\mstarn} |j|^{2a-2s}$ is bounded whereas  $\sum_{j>\mstarn}|j|^{-2p-2s}\lesssim n^{-1}$ and hence, $\drepn\sim n^{-1}$. To prove (ii) we make use of Corollary \ref{res:coro:lower}.  We observe that if $s-a\geq 0$ the sequence $\rw\Opw$ is bounded from below, and hence $\drwn\sim n^{-1}$. Otherwise, the condition $s-a<0$ implies $\drwn\sim n^{-(p+s)/(p+a)}$.

Proof of (pe). Note that $\mstarn$ satisfies $\mstarn \sim \log(n(\log n)^{-p/a})^{1/(2a)}$. In order to prove (i), we calculate that $\sum_{j>\mstarn}|j|^{-2p-2s}\sim(\log n)^{(-2p-2s+1)/(2a)}$ and  $n^{-1}\sum_{j=1}^{\mstarn}\exp(|j|^{2a}) |j|^{-2s}\lesssim(\log n)^{(-2p-2s+1)/(2a)}$.
In case (ii) we immediately obtain  $\drwn\sim (\log n)^{-(p+s)/a}$.

Proof of (ep). It holds true $\mstarn\sim\log(n(\log n)^{-a/p})^{1/(2p)}$.  Consider case (i).
If $s-a<1/2$ then $n^{-1}\sum_{j=1}^{\mstarn} |j|^{2a-2s}\sim  n^{-1}(\log n)^{(2a-2s+1)/(2p)}$.
If $s-a=1/2$ we conclude $n^{-1}\sum_{j=1}^{\mstarn} |j|^{2a-2s}\sim  n^{-1}\log(\log(n))$. On the other hand, the condition $s-a>1/2$ implies  that $\sum_{j=1}^{\mstarn} |j|^{2a-2s}$ is bounded and thus, we obtain the parametric rate $n^{-1}$.
Moreover, it is easily seen that $\sum_{j>\mstarn}|j|^{-2s}\exp(-|j|^{2p})\lesssim n^{-1}\sum_{j=1}^{\mstarn}|j|^{2a-2s}$.
In case (ii) if $s-a\geq 0$ then the sequence $\rw\Opw$ is bounded from below as mentioned above and thus, $\drwn\sim n^{-1}$. If $s-a< 0$ then $\drwn\sim  n^{-1}(\log n)^{(a-s)/p}$, which completes the proof.
\end{proof}
\subsection{Proofs of Section \ref{sec:adaptive}}\label{app:proofs:adaptive}
 At the end of this section we shall prove six technical Lemmata (\ref{app:part:l2} -- \ref{adapt:full:lem:1}) which are used in the following proofs. Let us introduce a nondecreasing sequence $\Delta:=(\Delta_m)_{m\geq 1}$  and its empirical analogon $\hDelta:=(\hDelta_m)_{m\geq 1}$ by $\Delta_{m}:=\max_{1\leq m'\leq m}\normV{[\rep]_{\um'}^t[\op]_{\um'}^{-1}}^2$ and  $\hDelta_{m}:=\max_{1\leq m'\leq m}\normV{[\rep]_{\um'}^t[\hop]_{\um'}^{-1}}^2$, respectively. Similarly to $\Mo$ introduced in \eqref{part:def:bounds:o} we define
\begin{equation}\label{part:def:bounds:u}
 \Mu:=\min\set{2\leq m\leq \Mh:\, 4D\upsilon_m^{-1}m^3 \max\limits_{1\leq j\leq m}[h]_j^2> a_n}-1
\end{equation}
where we set  $\Mu:=\Mh$ if the set is empty. Thus, $\Mu$ takes values between $1$ and $ \Mh$.
In the following $\mathcal C>0$ denotes a constant only depending on the classes $\Fswsr$, $\cTdDw$, the constants $\sigma$, $\eta$ and the representer $h$. For ease of notation, the value of $\mathcal C>0$ may change from line to line.

\begin{proof}[\textcolor{darkred}{\sc Proof of Theorem \ref{partially adaptive unknown operator}}]
The proof of the theorem is based on inequality \eqref{adap:main:ineq}. 
Observe that by Lemma \ref{app:upperboundmodel:l1} we have  $\Mu\leq M_n\leq \Mo$.
Due to condition $(\md)^3\max_{1\leq j\leq\md}[h]_j^2=o( a_n\opw_{\md})$ as $n\to\infty$ there exists $n_0 \geq 1$ only depending on $h$, $\sw$, and $\opw$ such that for all $n\geq n_0$ it holds
 $\md\leq\Mu$. We distinguish in the following the cases $n\geq n_0$ and $n<n_0$. First, consider $n\geq n_0$.
Applying Corollary \ref{app:part:c1}  together with estimate \eqref{adap:main:ineq} implies
\begin{equation*}
  \Ex \absV{\hell_{\tm}-\ell_\rep(\sol)}^2\leq \mathcal C\Big\{ \pen_{\md} +\bias_{\md} +   n^{-1}\Big\}.
\end{equation*}
From the definition of $\pen_m$ we infer $\pen_m\leq 24(3\sr+2\sigma^2)(1+\log n)n^{-1}\opD\sum_{j=1}^m\fou{\rep}_j^2\opw_j^{-1}$ since $\op\in\opwdD$, $U\in\cU_\sigma^\infty$, and $\sol\in\Fswsr$ . Moreover, since
$\sol\in\Fswsr$   and $\rep\in\cF_{1/\sw}$ estimate \eqref{app:gen:upper:l3:e2}   in Lemma \ref{app:gen:upper:l3} implies for all $1\leq m\leq\Mu$ that 
$\bias_m\leq \min_{1\leq m'\leq\Mu}2\,\sr\,\big\{\sum_{j>m'}[\rep]_j^2\sw_j^{-1}+
\opd\opD\opw_{m'}\sw_{m'}^{-1}\sum_{j=1}^{m'}[\rep]_j^2\opw_j^{-1}\big\}$.

Consequently,
\begin{equation*} 
 \Ex \absV{\hell_{\tm}-\ell_\rep(\sol)}^2\leq \mathcal C \Big\{ \max\Big(\sum_{j>\md}[\rep]_j^2\sw_j^{-1}, \ad\sum_{j=1}^\md\fou{\rep}_j^2\opw_j^{-1}\Big)+n^{-1}\Big\}.
\end{equation*}
Consider now $n<n_0$. Observe that for all $1\leq m\leq \Mh$ it holds
\begin{multline}\label{app:full:eq:6}
 |\hell_{m}-\ell_\rep(\sol)|^2\leq 2|[h]_{\um}^t[\hop]_{\um}^{-1}V_m|^2\1_{\Cset_m}+2 (|\ell_{\rep}(\sol_{m}-\sol)|^2+|\ell_\rep(\sol)|^2\1_{\Cset_m^c})\\
 \leq 2n\|[h]_{\uMh}\|^2\|V_{\Mh}\|^2+ 2 (|\ell_{\rep}(\sol_{m}-\sol)|^2+|\ell_\rep(\sol)|^2\1_{\Cset_m^c}).
\end{multline}

From the definition of $\Mh$ we infer $\|[h]_{\uMh}\|^2\leq [\rep]_1^2\,n^{5/4}$. Hence  inequality \eqref{app:gen:upper:l2:e2:1} in Lemma \ref{app:gen:upper:l2}, inequality \eqref{app:gen:upper:l3:e1} in Lemma \ref{app:gen:upper:l3} and Lemma \ref{adapt:full:lem:1} yield for all $\sol\in\Fswsr$   and $\rep\in\cF_{1/\sw}$
\begin{equation*}
 n\Ex|\hell_{\tm}-\ell_\rep(\sol)|^2
 \leq 2\,[\rep]_1^2\,n^{9/5} \|V_{\Mh}\|^2+ 6\sr\|\rep\|_{1/\sw}^2(1+Dd)n\leq  \mathcal C,
\end{equation*}
which proves the result.
\end{proof}

\begin{lem}\label{app:part:l1}
Consider $(\tpen_m)_{m\geq 1 }$ with $\tpen_m:=24\big(24\Ex[U^2]+96\eta^2\sr\, m^3\sw_m^{-1}\big)(1+\log n)n^{-1}$. Then under the conditions of Theorem \ref{partially adaptive unknown operator} we have  for all $n\geq 1$
\begin{equation*}\label{app:part:l1:eq:1}
\sup_{\op\in\cTdDw}\sup_{P_{U|W}\in\cU_\sigma^\infty}\Ex\max_{\md\leq m\leq\Mo} \vect{\absV{\hell_{m}-\ell_\rep(\sol_{m})}^2 -\frac{1}{6}\tpen_m }_+\leq \mathcal C\,n^{-1}.
  \end{equation*}
\end{lem}

\begin{proof}[\noindent\textcolor{darkred}{\sc Proof.}] 

 Similarly to the proof of Theorem \ref{res:upper} we obtain the decomposition
  \begin{multline*}
  |\hell_{m}-\ell_{\rep}(\sol_m)|^2\leq 2|[\rep]_{\um}^t[\op]_{\um}^{-1}V_m|^2
  +2m^{-1}\|[\rep]_{\um}^t[\op]_{\um}^{-1}\|^2\|V_m\|^2+\\
  \hfill 2n\|[\rep]_{\um}^t[\op]_{\um}^{-1}\Xi_m\|^2\|V_m\|^2
  \1_{\mho_m^c}+|\ell_\rep(\sol_m)|^2\1_{\Omega_m^c}.
 \end{multline*}

 Observe that 
 $\|[\rep]_{\um}^t[\op]_{\um}^{-1}\|^2\leq \Delta_m$ for all $m\geq 1$ and hence,  we have for all $\md\leq m\leq\Mo$ 
 \begin{multline*}
 \vect{\absV{\hell_{m}-\ell_\rep(\sol_{m})}^2 - \frac{1}{6}\tpen_m }_+
  \leq 2\Delta_m \Big(\frac{|[\rep]_{\um}^t[\op]_{\um}^{-1}V_m|^2}{\|[\rep]_{\um}^t[\op]_{\um}^{-1}\|^2}
    -\frac{\tpen_m}{24\Delta_m}\Big)_+\\\hfill
    +2\Delta_m\Big(\frac{\|V_m\|^2}{m}
    -\frac{\tpen_m}{24\Delta_m}\Big)_+
    +2n \Delta_m\|\Xi_m\|^2\|V_m\|^2\1_{\mho_m^c}
    +|\ell_\rep(\sol_m)|^2\1_{\Omega_m^c}\\
    =:I_m+II_m+III_m+IV_m.
 \end{multline*}
Consider the first two right hand side terms. We calculate
\begin{equation*}
 \Ex \max_{\md\leq m\leq\Mo} \vect{I_m+II_m}\leq 4\max_{\md\leq m\leq\Mo}\sup_{s\in\mathbb S^m}\Ex\Big(|s^tV_m|^2-\frac{\tpen_m}{24\Delta_m}\Big)_+\sum_{m=1}^{\Mo}\Delta_m.\label{app:part:l1:eq1}
 \end{equation*}
From the definition of $\tpen$ we infer for all $s\in\mathbb S^m$ and $\md\leq m\leq \Mo$ 
\begin{multline*}
 n\Ex\Big(|s^tV_m|^2-\frac{\tpen_m}{24\Delta_m}\Big)_+\leq 2 \Ex\Big((n^{-1/2}\sum_{i=1}^nU_is^t[f(W_i)]_{\um})^2-12\Ex[U^2](1+\log n)\Big)_+\\\hfill
    +2\Ex\Big((n^{-1/2}\sum_{i=1}^n(\sol(Z_i)-\sol_m(Z_i))s^t[f(W_i)]_{\um})^2-48\eta^2\sr\, m^3 \sw_m^{-1}(1+\log n)\Big)_+\\\hfill
  \leq C(\sigma, \eta, \sw,\sr, D)\, n^{-1}
\end{multline*}
where the last inequality follows from Lemma  \ref{app:part:l2} and \ref{app:part:l3}. Due to the definition of $\Mo$ and since $\Delta$ is nondecreasing we have $n^{-1}\sum_{m=1}^{\Mo}\Delta_m\leq D(n\opw_{\Mo})^{-1}(\Mo)^2 \max_{1\leq j\leq \Mo}[h]_j^2\leq 4D^2$. 
Consequently, $\Ex \max_{\md\leq m\leq\Mo}\vect{I_m+II_m}\leq \mathcal Cn^{-1}$.
Further, we obtain for $\sol\in\Fswsr$   and $\rep\in\cF_{1/\sw}$
 \begin{multline*}
 \Ex \max_{\md\leq m\leq\Mo}\vect{III_m}\leq n\Delta_{\Mo}\,(\Ex\|\Xi_{\Mo}\|^8)^{1/4}(\Ex\|V_{\Mo}\|^4)^{1/2}
    P^{1/4}\Big(\bigcup_{m=1}^{\Mo}\mho_m^c\Big)\\
\leq C(\sw)\,\eta^4(\sigma^2+(1+Dd)\sr)n^{-1}\Delta_{\Mo}(\Mo)^3P^{1/4}\Big(\bigcup_{m=1}^{\Mo}\mho_m^c\Big)
 \end{multline*}
where the last inequality is due to Lemma \ref{app:gen:upper:l2} and 
\begin{equation*}
 \Ex \max_{\md\leq m\leq\Mo} \vect{IV_m}\leq \sr\|\rep\|_{1/\sw}^2P\Big(\bigcup_{m=1}^{\Mo}\Omega_m^c\Big).\label{app:part:l1:eq2}
 \end{equation*}
Now applying $n^{-1}\Delta_{\Mo}(\Mo)^3\leq 4D^2$ and Lemma  \ref{app:part:l6} gives $\Ex \max_{\md\leq m\leq\Mo}\vect{III_m+IV_m}\leq \mathcal Cn^{-1}$, which completes the proof.
 \end{proof}

\begin{coro}\label{app:part:c1}
Under the conditions of Theorem \ref{partially adaptive unknown operator}  we have  for all $n\geq 1$
\begin{equation*}\label{app:part:l1:eq:1}
\sup_{\op\in\cTdDw}\sup_{P_{U|W}\in\cU_\sigma^\infty}\Ex\max_{\md\leq m\leq\Mo} \vect{\absV{\hell_{m}-\ell_\rep(\sol_{m})}^2 -\frac{1}{6}\pen_m }_+\leq \mathcal C \,n^{-1}.
  \end{equation*}
\end{coro}
\begin{proof}
Observe that $m^3\sw_m^{-1}=o(1)$ and $\|\sol-\sol_m\|_Z^2=o(1)$ as $m\to\infty$ due to Assumption \ref{ass:reg} and $T\in\cT_{d,D}^\opw$ (cf. proof of Corollary \ref{res:gen:coro:cons}), respectively. Thereby, there exists a constant $n_0$ only depending on $\sw$, $\sr$, and $\eta$ such that for all $n\geq n_0$ and $m\geq \md$ we have
\begin{equation}\label{app:part:c1:eq1}
 24\Ex[U^2]+96\eta^2\sr\, m^3\sw_m^{-1}\leq 72\big(\Ex[Y^2]+\|\sol_m\|_Z^2+\|\sol-\sol_m\|_Z^2\big)+96\eta^2\sr\, m^3\sw_m^{-1}\leq \varsigma_m^2.
\end{equation}
We distinguish in the following the cases $n<n_0$ and $n\geq n_0$. First, consider $n< n_0$.
Due to $n^{-1}\sum_{m=1}^{\Mo}\Delta_m\leq 4D^2$ and inequality \eqref{app:gen:upper:l2:e1:1} in Lemma \ref{app:gen:upper:l2} we calculate for all $s\in\mathbb S^m$
\begin{equation*}
 \sum_{m=1}^{\Mo}\Delta_m\Ex\Big(|s^tV_m|^2-\frac{\pen_m}{24\Delta_m}\Big)_+\leq \sum_{m=1}^{\Mo}\Delta_m \Ex|s^tV_m|^2\leq 8n_0D^2 \big( \sigma^2+C(\sw)\,\eta^2\, \normV{ \sol -   \sol_m}_\bw^2\big )\,n^{-1}.
\end{equation*}
Therefore, following line by line  the proof of Lemma \ref{app:part:l1} it is easily seen that it holds $n\Ex\max_{\md\leq m\leq\Mo} \big(\absV{\hell_{m}-\ell_\rep(\sol_{m})}^2 -\frac{1}{6}\pen_m \big)_+\leq \mathcal C$. Consider now $n\geq n_0$.
Inequality \eqref{app:part:c1:eq1} implies $\tpen_m\leq \pen_m$ and thus, $\big(\absV{\hell_{m}-\ell_\rep(\sol_{m})}^2 - \frac{1}{6}\pen_m \big)_+\leq \big(\absV{\hell_{m}-\ell_\rep(\sol_{m})}^2 - \frac{1}{6}\tpen_m \big)_+$ for all $\md\leq m\leq \Mo$. Thus, from Lemma \ref{app:part:l1} we infer $n\Ex\max_{\md\leq m\leq\Mo} \big(\absV{\hell_{m}-\ell_\rep(\sol_{m})}^2 -\frac{1}{6}\pen_m \big)_+\leq \mathcal C$, which completes the proof of the corollary.

\end{proof}

\begin{proof}[\textcolor{darkred}{\sc Proof of Theorem \ref{fully adaptive unknown operator}}]
Similarly to the proof of Theorem \ref{partially adaptive unknown operator} and since $\hpen$ is a nondecreasing sequence we have for all $1\leq m \leq\hM$
\begin{equation*}  \label{app:full:eq:4}
  \absV{\hell_{\hm}-\ell_\rep(\sol)}^2\lesssim\hpen_{m} +\bias_m+ \max_{m\leq m'\leq \hM} \vect{\absV{\hell_{m'}-\ell_\rep(\sol_{m'})}^2 -\frac{1}{6}\hpen_{m'}}_+.
\end{equation*}
Let us introduce the set
\begin{multline*}
  \cA:=\set{ \pen_{m}\leq\hpen_{m}\leq 8 \pen_{m}, \quad 1\leq m\leq\Mo}\cap\{\Mu\leq \hM\leq \Mo\},
\end{multline*}
then we conclude for all $1\leq m \leq \Mu$
\begin{equation*}
  \absV{\hell_{\hm}-\ell_\rep(\sol)}^2\1_{\cA}\lesssim \pen_{m} + \bias_m+ \max_{m\leq m'\leq \Mo} \vect{\absV{\hell_{m'}-\ell_\rep(\sol_{m'})}^2 -\frac{1}{6}\pen_{m'}}_+.
\end{equation*}

Thereby, similarly as in the proof of Theorem \ref{partially adaptive unknown operator} we obtain for all $\sol\in\Fswsr$   and $\rep\in\cF_{1/\sw}$ the upper bound for all $n\geq 1$
\begin{equation}\label{app:full:eq:5}
    \Ex \absV{\hell_{\hm}-\ell_\rep(\sol)}^2\1_{\cA}\leq
    \mathcal C\,\drepnl.
\end{equation}
Let us now evaluate the risk of the adaptive estimator $\hell_{\hm}$ on $\cA^c$. 
From the definition of $\Mh$ we infer $\|[h]_{\uMh}\|^2\leq [\rep]_1^2\,n\,\Mh$. Consequently, inequality \eqref{app:full:eq:6} together with \eqref{app:gen:upper:l2:e2:1} in Lemma \ref{app:gen:upper:l2}, \eqref{app:gen:upper:l3:e1} in Lemma \ref{app:gen:upper:l3} and Lemma \ref{adapt:full:lem:1} yields  for all $\sol\in\Fswsr$   and $\rep\in\cF_{1/\sw}$
\begin{equation*}\label{app:full:eq:6}
 \Ex|\hell_{\hm}-\ell_\rep(\sol)|^2\1_{\cA^c}
 \leq 2\,[\rep]_1^2\,n^2\Mh (\Ex\|V_{\Mh}\|^4)^{1/2}P(\cA^c)^{1/2}+ 6\sr\|\rep\|_{1/\sw}^2(1+Dd) P(\cA^c)\leq  \mathcal C\,n^{-1}.
\end{equation*}
The result follows by combining the last inequality with \eqref{app:full:eq:5}. 
\end{proof}

\paragraph{Technical assertions.}\hfill\\[1ex]
The following paragraph gathers technical results used in the proofs of Section  \ref{sec:adaptive}.
In the following we denote $\xi_s(w):=\sum_{j=1}^ms_jf_j(w)$ where $s\in\mathbb S^m=\{s\in\mathbb R^m:\|s\|=1\}$.

\begin{lem}\label{app:part:l2}
Let Assumptions \ref{ass:A1} and \ref{ass:A2} hold.
Then for all $n\geq 1$ and $1\leq m\leq\lfloor n^{1/4}\rfloor$ we have
\begin{equation*}
\sup_{P_{U|W}\in\cU_\sigma^\infty}\sup_{s\in\mathbb S^m} \Ex\Big[\Big(\frac{1}{n}\Big|\sum_{i=1}^nU_i\xi_s(W_i)\Big|^2-12\Ex[U^2](1+\log n)\Big)_+\Big]\leq C(\sigma, \eta)\, n^{-1}.
\end{equation*}
\end{lem}
\begin{proof}
Let us denote $\delta=12\Ex[U^2](1+\log n)$. Since the error term $U$ satisfies Cramer's condition we may apply Bernstein's inequality and since $\Ex[U^2|W]\leq\sigma^2$ we have 
\begin{multline}\label{app:part:l2:eq1}
 \Ex\Big[\Big(\frac{1}{n}\Big|\sum_{i=1}^nU_i\xi_s(W_i)\Big|^2-\delta\Big)_+|W_1,\dots,W_n\Big]\\\hfill
 =\int_0^\infty P\Big(\sum_{i=1}^nU_i\xi_s(W_i)\geq \sqrt{n(t+\delta)}|W_1,\dots,W_n\Big)dt\\
 \leq \int_0^\infty\exp\Big(\frac{-n(t+\delta)}{8 \sigma^2\sum_{i=1}^n|\xi_s(W_i)|^2}\Big)dt
 + \int_0^\infty\exp\Big(\frac{-\sqrt{n(t+\delta)}}{4 \sigma \max_{1\leq i\leq n}|\xi_s(W_i)|}\Big)dt.
\end{multline} 
Consider the first summand of \eqref{app:part:l2:eq1}.
Let us introduce the set
\begin{equation*}
\cB:=\set{\forall 1\leq j,l\leq m:\,|n^{-1}\sum_{i=1}^nf_j(W_i)f_l(W_i)-\delta_{jl}|\leq \frac{\log n}{3\sqrt n}}
\end{equation*}
where $\delta_{jl}=1$ if $j=l$ and zero otherwise.
Applying Cauchy-Schwarz's inequality twice we observe on $\cB$ for all $n\geq 1$ and $1\leq m\leq \Mo$
\begin{equation*}
| n^{-1}\sum_{i=1}^n|\xi_s(W_i)|^2-1|\1_\cB\leq \sum_{j,l=1}^m|z_j| |z_l| |n^{-1}\sum_{i=1}^nf_j(W_i)f_l(W_i)-\delta_{jl}|\1_\cB
 \leq \frac{1}{2}
\end{equation*}
since $n^{-1/4}\log n\leq 3/2$ for all $n\geq 1$. Thereby, it holds $ n^{-1}\sum_{i=1}^n|\xi_s(W_i)|^2\1_\cB\leq3/2$   and thus, 
\begin{multline}\label{app:part:l2:eq2}
 n\Ex\Big[\int_0^{\infty}\exp\Big(\frac{-n(t+\delta)}{8 \sigma^2\sum_{i=1}^n|\xi_s(W_i)|^2}\Big)dt\1_\cB\Big]
 \leq 12\sigma^2\exp\Big(\log n-\frac{\delta}{12 \sigma^2}\Big)
 \leq 6\sigma^2.
\end{multline}
On the complement of $\cB$ observe that $\sup_{j,l}\Var (f_j(W)f_l(W))<\eta^2$ due that Assumption \ref{ass:A1} (i) and thus, Assumption \ref{ass:A2} together with Bernstein's inequality yields
\begin{multline*}
 P(\cB^c)\leq \sum_{j,l=1}^mP\Big(3\big|\sum_{i=1}^nf_j(W_i)f_l(W_i)-\delta_{jl}\big|>\sqrt n \log n\Big)\\
 \leq 2m^2\exp\Big(-\frac{ n(\log n)^2}{36 n \eta^4+6\eta\sqrt n\log n}\Big)
\leq 2 \exp\Big(2\log m-\frac{(\log n)^2}{42\eta^4}\Big).
\end{multline*}

By Assumption \ref{ass:A1} (i) it holds $\Ex|\xi_s(W)|^4\leq \Ex|\sum_{j=1}^mf_j^2(W)|^2\leq m^2\eta^4$.
 Thereby
\begin{equation}\label{app:part:l2:eq2b}
 n\Ex\Big[\int_0^{\infty}\exp\Big(\frac{-n(t+\delta)}{8 \sigma^2\sum_{i=1}^n|\xi_s(W_i)|^2}\Big)dt\1_{\cB^c}\Big]
 \leq 8\sigma^2n\big(\Ex|\xi_s(W_1)|^4 P(\cB^c)\big)^{1/2}
 \leq 12\sigma^2\eta^2
\end{equation}
for all $n\geq \exp(126\eta^4)$ and $1\leq m\leq \lfloor n^{1/4}\rfloor$. For $n< \exp(126\eta^4)$ it holds $n\Ex[|\xi_s(W_1)|^2\1_{\cB^c}]< \exp(126\eta^4)$.
Consider the second summand of \eqref{app:part:l2:eq1}. 
Since $\exp(-1/x)$, $x>0$, is a concave function and $\Ex|\xi_s(W)|^4\leq  m^2\eta^4$ we deduce for all $1\leq m\leq \lfloor n^{1/4}\rfloor$
\begin{multline}\label{app:part:l2:eq3}
  \Ex\Big[\int_{0}^\infty\exp\Big(\frac{-\sqrt{n(t+\delta)}}{4 \sigma \max_{1\leq i\leq n}|\xi_s(W_i)|}\Big)dt\Big]
\leq \int_0^\infty\exp\Big(\frac{-\sqrt{n(t+\delta)}}{4 \sigma \Ex\max_{1\leq i\leq n}|\xi_s(W_i)|}\Big)dt\\
\leq \int_0^\infty\exp\Big(\frac{-\sqrt{n(t+\delta)}}{4 \sigma (n\Ex|\xi_s(W)|^4)^{1/4}}\Big)dt
\leq \int_0^\infty\exp\Big(\frac{-n^{1/4}\sqrt{(t+\delta)}}{4 \sigma\,\eta\sqrt m}\Big)dt\\
  \leq 8\sigma\,\eta \sqrt{m/n}\exp\Big(\frac{-n^{1/4}\sqrt\delta}{4 \sigma\,\eta\sqrt m}\Big)
  \Big(n^{1/4}\sqrt\delta+4\sigma\,\eta\sqrt m\Big)
\leq C(\sigma, \eta)n^{-1}.
\end{multline}
The assertion follows now by combining inequality \eqref{app:part:l2:eq1} with \eqref{app:part:l2:eq2}, \eqref{app:part:l2:eq2b}, and \eqref{app:part:l2:eq3}.
\end{proof}

\begin{lem}\label{app:part:l3}
Let Assumptions \ref{ass:reg} and \ref{ass:A1} hold. Then for all $n\geq 1$ and $m\geq 1$ we have
 \begin{equation*}
  \sup_{\op\in\cTdDw} \sup_{s\in\mathbb S^m}\Ex\Big[\Big(\frac{1}{n}\Big|\sum_{i=1}^n(\sol(Z_i)-\sol_m(Z_i))\xi_s(W_i)\Big|^2-48\eta^2\sr\frac{m^3}{\sw_m}(1+\log n)\Big)_+\Big]\leq C(\eta, \sw,\sr,D) n^{-1}.
\end{equation*}
\end{lem}
\begin{proof}
Let us consider a sequence $w:=(w_j)_{j\geq 1}$ with $w_j:=j^2$. Since $[\op(\sol-\sol_m)]_\um=0$ we conclude for $m\geq 1$, $s\in\mathbb S^m$, and $k=2,3,\dots$ that
 \begin{multline*}
  \Ex|(\sol(Z)-\sol_m(Z))\xi_s(W)|^k=\Ex|\sum_{l=1}^\infty[\sol-\sol_m]_l\sum_{j=1}^ms_j(e_l(Z)f_j(W)-[T]_{jl})|^k\\\hfill
\leq \|\sol-\sol_m\|_w^k \Ex|\sum_{l=1}^\infty w_l^{-1}\sum_{j=1}^m(e_l(Z)f_j(W)-[T]_{jl})^2|^{k/2}\\
\leq \|\sol-\sol_m\|_w^k m^{k/2}(\pi/\sqrt6)^{k}\sup_{j,l\in\mathbb N}\Ex|e_l(Z)f_j(W)-[T]_{jl}|^{k}
 \end{multline*}
where due to Assumption \ref{ass:A1} (i)  $\sup_{j,l\in\mathbb N}\Var(e_l(Z)f_j(W))\leq \eta^2$ and due to
Assumption \ref{ass:A1} (ii) it holds $\sup_{j,l\in\mathbb N}\Ex|e_l(Z)f_j(W)-[T]_{jl}|^{k}\leq k!\eta^k$ for $k\geq 3$.
 Moreover, similarly to the proof of  \eqref{app:gen:upper:l3:e1} in Lemma \ref{app:gen:upper:l3} we conclude
 $ m^{k/2}\|\sol-\sol_m\|_w^k\leq (m^3\sw_m^{-1})^{k/2}(2+2Dd)^{k/2}\sr^{k/2}$. Let us denote $\mu_m:=\eta\,(1+Dd)\sqrt{6\sr \, m^3\sw_m^{-1}}$.
Consequently, for all $m\geq 1$ we have $\Ex|(\sol(Z)-\sol_m(Z))\xi_s(W)|^2\leq \mu_m^2 $ and 
\begin{equation}\label{app:part:l3:cram}
  \sup_{s\in\mathbb S^m}\Ex|(\sol(Z)-\sol_m(Z))\xi_s(W)|^k\leq \mu_m^k k!\text{ for }k=3,4,\dots.
 \end{equation}
Now Bernstein's inequality gives for all $m\geq 1$
\begin{multline*}
\sup_{s\in\mathbb S^m} \Ex\Big[\Big(\frac{1}{n}\Big|\sum_{i=1}^n(\sol(Z_i)-\sol_m(Z_i))\xi_s(W_i)\Big|^2-8\mu_m^2(1+\log n)\Big)_+\Big]\\\hfill
 \leq2 \int_0^\infty\exp\Big(\frac{-(t+\delta)}{8\mu_m^2}\Big)dt+
2 \int_0^\infty\exp\Big(\frac{-\sqrt{n(t+\delta)}}{4\mu_m}\Big)dt\\\hfill
\leq 16\mu_m^2\exp(-\log n)+16\mu_m n^{-1/2}\exp\Big(\frac{-\sqrt{n(1+\log n)}}{2}\Big)(4\mu_m+\sqrt{8n\mu_m^2(1+\log n)})\\
\leq C(\eta, \sw,\sr,D)n^{-1}
\end{multline*}
and thus, the assertion follows.

\end{proof}

 \begin{lem}\label{app:part:l6}
Let $T\in\cT_{d,D}^\opw$. Then for all $n\geq 1$  it holds
 \begin{gather}
  P\big(\bigcup_{m=1}^{\Mo} \mho_m^c\big)\leq C(\rep,\opw,\eta, D) n^{-4}, \label{app:part:l6:c:1}\\
  P\big(\bigcup_{m=1}^{\Mo} \Omega_m^c\big)\leq C(\rep,\opw,\eta, D)\, n^{-1}.\label{app:part:l6:c:2}
 \end{gather}
\end{lem}
\begin{proof}
 Proof of \eqref{app:part:l6:c:1}.
Since $T\in\cT_{d,D}^\opw$ we have $\|[\op]_{\um}^{-1}\|^2\leq D \opw_m^{-1}$ and thus, exploiting Lemma \ref{pr:minimax:l1} together with the definition of $\Mo$ gives
\begin{multline*}
  n^4P\Big(\bigcup_{m=1}^{\Mo} \mho_m^c\Big)\leq 
   2\exp\Big(-\frac{1}{48\eta D}\frac{n\opw_{\Mo}}{(\Mo)^3}+3\log \Mo+4\log n\Big)\leq C(\rep,\opw,\eta, D).
\end{multline*}
  Proof of \eqref{app:part:l6:c:2}. Due to the definition of $\Mo$ there exists some $n_0\geq 1$ such that $n\geq 4D\opw_{\Mo}^{-1}$ for all $n\geq n_0$. Thereby, condition $T\in\cT_{d,D}^\opw$ implies $\max_{1\leq m\leq \Mo}\|[\op]_{\um}^{-1}\|^2\leq D\opw_{\Mo}^{-1} \leq n/4$ for all $n\geq n_0$. This gives $\bigcup_{m=1}^{\Mo} \Omega_m^c\subset\bigcup_{m=1}^{\Mo} \mho_m^c$ and inequality \eqref{app:part:l6:c:2} follows by making use of \eqref{app:part:l6:c:1}. If $n<n_0$ then $nP\big(\bigcup_{m=1}^{\Mo} \Omega_m^c\big)\leq n_0$ and the assertion follows since $n_0$ only depends on $\rep$, $\opw$ and $D$.
\end{proof}

\begin{lem}\label{app:upperboundmodel:l1}Let $T\in\cT_{d,D}^\opw$. Then it holds  $\Mu\leq M_n\leq \Mo$ for all $n\geq 1$.\end{lem}
\begin{proof}[{\sc Proof.}]
Consider $\Mu\leq M_n$.  If $\Mu=1$ or $M_n=\Mh$ the result is trivial.   If $M_n=1$, then clearly $\Mu=1$. It remains to consider  $\Mu>1$ and
  $\Mh>M_n>1$. Due to $T\in\cT_{d,D}^\opw$ it holds $\normV{\fou{\op}_{\underline{M_n+1}}^{-1}}^{-2}\geq D^{-1}\opw_{M_n+1}$ and thus, by the definition of $M_n$ and $\Mu$ it is easily seen that
  \begin{equation*}
\frac{\Opw_{\Mu}}{\max\limits_{1\leq j\leq\Mu}[h]_j^2 (\Mu)^3}> \frac{4 \Opw_{M_n+1}}{\max\limits_{1\leq j\leq M_n+1}[\rep]_j^2(M_n+1)^3},
\end{equation*}
 and thus, $M_n+1 > \Mu$, i.e. $M_n\geq \Mu$.
Consider $ M_n\leq \Mo$.  If $M_n=1$ or $\Mo=\Mh$ the result is trivial, while otherwise since $\opw_m^{-1}\leq \normV{\fou{\op}_{\um}^{-1}}^{2}\sup_{\|E_m\phi\|_\opw=1}\|F_m\op E_m\phi\|^2
\leq D\normV{\fou{\op}_{\um}^{-1}}^{2}$ due to condition $\op\in\cTdw$ with $d\leq D$ and by the definition of $M_n$ and $\Mo$ it follows 
  \begin{equation*}
\frac{\Opw_{M_n}}{\max\limits_{1\leq j\leq M_n}[\rep]_j^2M_n^3}> \frac{4\Opw_{\Mo+1}}{\max\limits_{1\leq j\leq \Mo+1}[\rep]_j^2(\Mo+1)^3}.
\end{equation*}
Thus,  $\Mo+1> M_n$, i.e. $\Mo\geq M_n$, which completes the proof.
\end{proof}

In the following, we make use of the notation $\sigma_Y^2:=\Ex[Y^2]$ and $\widehat \sigma_Y^2:=n^{-1}\sum_{i=1}^nY_i^2$. 
Further, let us introduce the events
\begin{gather}
 \cH:=\set{\|\Xi_m\|\|[\op]_{\um}^{-1}\|\leq 1/4  \quad\forall\,1\leq m\leq (\Mo+1)},\\
  \cG:=\set{\sigma_Y^2\leq 2\,\widehat \sigma_Y^2\leq 3\,\sigma_Y^2},\\
\cJ:=\set{\|[\op]_\um^{-1}V_m\|^2\leq \frac{1}{8}\big(\|[\op]_\um^{-1}[g]_\um\|^2+\sigma_Y^2\big)\quad \forall\,1\leq m\leq\Mo}.
\end{gather}

\begin{lem}\label{adapt:full:lem:0}
Let $T\in\cT_{d,D}^\opw$. Then it holds $ \cH\cap\cG\cap\cJ\subset\cA$.
 \end{lem}

 \begin{proof}
For all $1\leq m\leq \Mo$ observe that condition $\|\Xi_m\|\|[\op]_{\um}^{-1}\|\leq 1/4$ yields by the usual Neumann series argument that $\|([I]_{\um}+\Xi_m[T]_{\um}^{-1})^{-1}-[I]_{\um}\|\leq 1/3$. Thus, using the identity $[\hop]_{\um}^{-1} = [\op]_{\um}^{-1} - [\op]_{\um}^{-1}\big(([I]_{\um}+\Xi_m[T]_{\um}^{-1})^{-1}-[I]_{\um}\big)$ we conclude
  \begin{gather*}
  2\|[\rep]_{\um}^t[\op]_{\um}^{-1}\|\leq 3\|[\rep]_{\um}^t[\hop]_{\um}^{-1}\|\leq 4\|[\rep]_{\um}^t[\op]_{\um}^{-1}\|.
\end{gather*}
Similarly, we have $ 2\|[\op]_{\um}^{-1}v_m\|\leq3 \|[\hop]_{\um}^{-1}v_m\|\leq 4\|[\op]_{\um}^{-1}v_m\|$ for all $v_m\in\mathbb R^m$.
Thereby, since $[\hop]_\um^{-1} V_m=[\hop]_\um^{-1} [\hgf]_\um-[\op]_\um^{-1}[g]_\um$ we conclude
  \begin{gather*}
  \|[\op]_{\um}^{-1}[g]_{\um}\|^2\leq (32/9)\|[\op]_\um^{-1} V_m\|^2 + 2\|[\hop]_\um^{-1}[\hgf]_{\um}\|^2,\\
  \|[\hop]_{\um}^{-1}[\hgf]_{\um}\|^2\leq (32/9)\|[\op]_\um^{-1} V_m\|^2 + 2\|[\op]_\um^{-1}[g]_{\um}\|^2.
\end{gather*}
On $\cJ$ it holds $\|[\op]_\um^{-1} V_m\|^2 \leq \frac{1}{8}(\|[\op]_\um^{-1}[g]_{\um}\|^2 + \sigma_Y^2)$. Thereby, the last two inequalities imply
\begin{gather*}
  (5/9)(\|[\op]_{\um}^{-1}[g]_{\um}\|^2 + \sigma_Y^2)\leq \sigma_Y^2 + 2\|[\hop]_\um^{-1}[\hgf]_{\um}\|^2,\\
  \|[\hop]_{\um}^{-1}[\hgf]_{\um}\|^2\leq (22/9)\|[\op]_\um^{-1}[g]_{\um}\|^2 +(4/9)\sigma_Y^2.
\end{gather*}
On $\cG$ it holds $\sigma_Y^2\leq 2\widehat \sigma_Y^2\leq 3\sigma_Y^2$ which gives
\begin{gather*}
  (5/9)(\|[\op]_{\um}^{-1}[g]_{\um}\|^2 + \sigma_Y^2)\leq (3/2)\widehat \sigma_Y^2+ 2\|[\hop]_\um^{-1}[\hgf]_{\um}\|^2,\\
  \|[\hop]_{\um}^{-1}[\hgf]_{\um}\|^2+\widehat \sigma_Y^2\leq (22/9)\|[\op]_\um^{-1}[g]_{\um}\|^2 +(10/9)\sigma_Y^2.
\end{gather*}
Combing the last two inequalities we conclude for all $1\leq m\leq \Mo$
\begin{equation*}
  (5/18)\big(\|[\op]_{\um}^{-1}[g]_{\um}\|^2 + \sigma_Y^2\big)\leq \|[\hop]_{\um}^{-1}[\hgf]_{\um}\|^2+\widehat \sigma_Y^2\leq (22/9)\big(\|[\op]_\um^{-1}[g]_{\um}\|^2 +\sigma_Y^2\big).
\end{equation*}
Consequently, we have
\begin{equation*}
 \cH\cap\cG\cap\cJ\subset\set{ 4\,\Delta_m\leq9\,\hDelta_m\leq 16\,\Delta_m\text{ and } 5\,\varsigma_m^2\leq 18\,\widehat\varsigma_m^2\leq 44\,\varsigma_m^2\quad\forall 1\leq m\leq\Mo}
\end{equation*}
and thus, $\cH\cap\cG\cap\cJ\subset\set{ \pen_m\leq \hpen_m\leq 18\pen_m \,\forall 1\leq m\leq\Mo}$.
Moreover, it holds
 $\cH\subset\big\{\Mu\leq\hM\leq \Mo\big\}$, which can be seen as follows. Consider $\{\hM<\Mu\}$. In case of $\hM=\Mh$ or $\Mu=1$ clearly $\{\hM<\Mu\}=\emptyset$.
 Otherwise by the definition of $\hM$ it holds 
 \begin{equation*}
 \{\hM<\Mu\}=\bigcup_{m=1}^{\Mu-1}\set{\hM =m}\subset\Big\{\exists 2\leq m\leq \Mu : m^3\normV{\fou{\hop}_{\um}^{-1}}^{2} \max\limits_{1\leq j\leq m}[h]_j^2> a_n\Big\}.  
 \end{equation*}
 By the definition of $\Mu$ and the property $\normV{\fou{\op}_{\um}^{-1}}^{2}\leq D\opw_m^{-1}$ there  exists   $2\leq m\leq \Mu$ such that on $\{\hM<\Mu\}$ it holds $\normV{\fou{\hop}_{\um}^{-1}}^{2}>4D\opw_m^{-1}\geq4\,\normV{\fou{\op}_{\um}^{-1}}^{2}$ and thereby,
\begin{align}
\label{diss:eq:13}
\set{\hM < \Mu}&\subset
\set{\exists 2\leq m\leq \Mu:\normV{\fou{\hop}_{\um}^{-1}}^{2}\geq4\,\normV{\fou{\op}_{\um}^{-1}}^{2}}.
\end{align}
Consider $\{\hM > \Mo\}$. In case of and $\hM=\Mh$ or $\Mu=1$ clearly $\{\hM<\Mu\}=\emptyset$. Otherwise, condition $T\in\cT_{d}^\opw$ with $d\leq D$ implies $\opw_m^{-1}\leq D\normV{\fou{\op}_{\um}^{-1}}^{2}$  as seen in the proof of Lemma \ref{app:part:l6}. Thereby, we conclude similarly as above
\begin{align}\label{diss:eq:14}
\set{\hM > \Mo}\subset
\set{\normV{\fou{\op}_{\underline{\Mo+1}}^{-1}}^{2}\geq4\normV{\fou{\hop}_{\underline{\Mo+1}}^{-1}}^{2}}.
\end{align}
Again applying the Neumann series argument we observe
\begin{equation*}\label{diss:eq:15}
\cH\subset  \set{\forall\,1\leq m\leq (\Mo+1) :
  2\normV{\fou{\Op}_\um^{-1}}\leq 3\normV{\fou{\hOp}_\um^{-1}}\leq 4\normV{\fou{\Op}_\um^{-1}}},
\end{equation*}
which combined  with  \eqref{diss:eq:13} and \eqref{diss:eq:14}
 yields $\big\{\Mu\leq\hM\leq \Mo\big\}^c\subset\cH^c$ and thus, completes the proof.
\end{proof}

\begin{lem}\label{adapt:full:lem:1}
Under the conditions of Theorem \ref{fully adaptive unknown operator} we have for all $n\geq 1$
\begin{equation*}
 n^4(\Mh)^4P(\cA^c)\leq \mathcal C.
\end{equation*}
 \end{lem}
\begin{proof}
Due to  Lemma \ref{adapt:full:lem:0} it holds  $n^4(\Mh)^4P(\cA^c)\leq n^4(\Mh)^4\set{P(\cH^c)+P(\cJ^c)+P(\cG^c)}$. Therefore, the assertion follows if the right hand side is bounded by a constant $\mathcal C$, which we prove in the following. Consider $\cH$. 
 From condition $T\in\cT_{d,D}^\opw$ and  Lemma \ref{pr:minimax:l1} we infer
 \begin{equation}\label{adapt:full:lem:1:eq2}
n^4(\Mh)^4P(\cH^c)
\leq 2\exp\Big(-\frac{1}{128D\eta} \frac{ n\opw_{\Mo+1}}{(\Mo+1)^2}+3\log(\Mo+1) + 5\log n\Big)\leq C(\rep,\opw,\eta, D)
\end{equation} 
where the last inequality is  due to condition $(\Mo+1)^2\log n=o(n\opw_{\Mo+1})$.
Consider $\cG$. Due to condition $m^3\sw_m^{-1}=o(1)$ as $m\to\infty$ and $U\in \cU_\sigma^\infty$ we observe
$\Ex[Y^{k}]\leq 2^{k}(\Ex[\sol^{k}(Z)]+\Ex[U^k])\leq C(\sw,\sr, \sigma)\sup_{j\geq 1}\Ex[e_j^k(Z)]$. Thereby, assumption $\sup_{j\geq 1}\Ex[e_j^{20}(Z)]\leq \eta^{20}$ together with
 Theorem 2.10 in \cite{Petrov1995} imply
\begin{multline}\label{adapt:full:lem:1:eq3}
n^4(\Mh)^4P(\cG^c)\leq n^5 P\big(|\widehat\sigma_Y^2-\sigma_Y^2|>\sigma_Y^2/2\big)\leq 1024\,\sigma_Y^{-20}n^5\Ex\big|n^{-1}\sum_{i=1}^nY_i^2-\sigma_Y^2\big|^{10}\\
\leq 1024\,\sigma_Y^{-20}\Ex|Y^2-\sigma_Y^2|^{10}\leq C(\sw,\sr, \sigma,\eta).
\end{multline}
Consider $\cJ$. For all $m\geq 1$ observe that the centered random variables $(Y_i-\sol(Z_i))f_j(W_i)$, $1\leq i\leq n$, satisfy Cramer's condition \eqref{app:part:l3:cram} with $\mu_m=\eta\,(1+Dd)\sqrt{6\sr\,m^3\sw_m^{-1}}\leq C(\eta, \sw,\sr,D)$.
From \eqref{app:gen:upper:l3:e1} in  Lemma \ref{app:gen:upper:l3}, $\sol\in\cF_\sw^\sr$, and $P_{U|W}\in\cU_\sigma^\infty$ we infer $\|\sol_m\|_Z^2+\sigma_Y^2\leq 4(2+Dd)\sr+2\sigma^2$.
Moreover, it holds $\normV{[\op]_\um^{-1}V_m}^2\leq D\opw_m^{-1}\normV{V_m}^2$ by employing condition $T\in\cT_{d,D}^\opw$.
Now Bernstein's inequality yields for all $1\leq m\leq \Mo$ 
\begin{multline*}
n^6 P\Big(\|[\op]_\um^{-1}V_m\|^2> (\|[\op]_\um^{-1}[g]_\um\|^2+\sigma_Y^2)/8\Big)\\\hfill
\leq n^6\sum_{j=1}^m P\Big(\Big|\sum_{i=1}^n(Y_i-\sol(Z_i))f_j(W_i)\Big|^2> \frac{n^2\opw_m}{8Dm}\big(\|\sol_m\|_Z^2+\sigma_Y^2\big)\Big)\\\hfill
\leq2 \,n^6m\,\exp\Big(-\frac{n^2\opw_mm^{-1}(\|\sol_m\|_Z^2+\sigma_Y^2)}{32Dn\mu_m^2 + 16\mu_mn\opw_m^{1/2}m^{-1/2}(\|\sol_m\|_Z^2+\sigma_Y^2)^{1/2}}\Big)\\
\leq 2\exp\Big(7\log n-\frac{n\,\opw_{\Mo}\,\sigma_Y^2}{\Mo\,C(\sigma, \eta, \sw,\sr,D)}\Big).
\end{multline*}
Due to the definition of $\Mo$ the last estimate implies $n^4(\Mh)^4 P(\cJ^c)\leq \mathcal C$, which completes the proof.
\end{proof}

\end{document}